\documentclass[12pt, a4paper]{article}

\usepackage[latin1]{inputenc}
\usepackage[english]{babel}
\usepackage{amssymb}
\usepackage{amsfonts}
\usepackage{amsmath}
\usepackage{amsthm}
\usepackage{epsfig}
\usepackage{graphics}
\usepackage{dsfont}
\usepackage[usenames, dvipsnames]{xcolor}
\usepackage{verbatim}
\usepackage{latexsym}
\usepackage{setspace}

\theoremstyle{plain}
\newtheorem{theorem}{Theorem}[section]
\newtheorem{corollary}[theorem]{Corollary}
\newtheorem{lemma}[theorem]{Lemma}
\newtheorem{proposition}[theorem]{Proposition}

\theoremstyle{definition}

\theoremstyle{definition}
\newtheorem{remark}[theorem]{Remark}

\newtheorem{example}[theorem]{Example}  
 
\allowdisplaybreaks[4]

\newcommand{\NN}{\mathbb{N}}

\newcommand{\RR}{\mathbb{R}}
\newcommand{\PP}{\mathbb{P}}

\newcommand{\FF}{\mathbb{F}}

\newcommand{\TT}{\mathbb{T}}
\newcommand{\EE}{\mathbb{E}}

\newcommand{\cS}{\mathcal{S}}

\newcommand{\cF}{\mathcal{F}}

\newcommand{\cE}{\mathcal{E}}

\newcommand{\supp}{\mathrm{supp}\,}

\newcommand{\Log}{\mathrm{Log}}

\newcommand{\frace}{\mathfrak{E}}

\begin{document}
	\title{Exponential functionals of Markov additive processes} 
	\author{Anita Behme\thanks{Technische Universit\"at
			Dresden, Institut f\"ur Mathematische Stochastik, Zellescher Weg 12-14, 01069 Dresden, Germany, \texttt{anita.behme@tu-dresden.de} and \texttt{apostolos.sideris@tu-dresden.de}, phone: +49-351-463-32425, fax:  +49-351-463-37251.}\; and Apostolos Sideris$^\ast$}
	\date{\today}
	\maketitle
	
	\vspace{-1cm}
	\begin{abstract}
		We provide necessary and sufficient conditions for convergence of exponential integrals of Markov additive processes. By contrast with the classical Lévy case studied by Erickson and Maller we have to distinguish between almost sure convergence and convergence in probability. Our proofs rely on recent results on perpetuities in a Markovian environment by Alsmeyer and Buckmann.
	\end{abstract}
	
	

	\section{Introduction}\label{S0}
	\setcounter{equation}{0}
	
	Given a bivariate L\'evy process $(\xi_t, \eta_t)_{t\geq 0}$ the corresponding {\sl exponential functional} is defined as
	\begin{equation}\label{eq:expfunclevy}
	\int_{(0,\infty)} e^{-\xi_{t-}} d\eta_t, 
	\end{equation}
	provided that the integral converges a.s. Necessary and sufficient conditions for this convergence in terms of the L\'evy characteristics of $(\xi_t, \eta_t)_{t\geq 0}$ have been given in \cite[Thm. 2]{ericksonmaller05}.
	
	As shown in \cite{lindnermaller05} exponential functionals of L\'evy processes describe exactly the stationary distributions of generalized Ornstein-Uhlenbeck processes, a class of processes that stems from physics, and nowadays has numerous applications e.g. in finance and insurance, see e.g. \cite{KLM:2004, paulsen93}. Due to this connection, the resulting importance in applications, and their complexity, exponential functionals have gained a lot of attention from various researchers over the last decades, see e.g. \cite{behmeALEA, behmelindnerJOTP, bertoinyor, kuznetsovetal, PardoPatieSavov, PardoRiveroSchaik} to name just a few. 
	
	Dating back to \cite{hamilton} Markov switching models have become a popular tool in financial mathematics and elsewhere. Thus it is a natural attempt to study exponential functionals with a Markov switching behaviour. In our paper, given a bivariate Markov additive process $(\xi_t,\eta_t,J_t)_{t\geq 0}$ with Markovian component $(J_t)_{t\geq 0}$, we denote the \emph{exponential integral of the Markov additive process $(\xi_t,\eta_t,J_t)_{t\geq 0}$} by
	\begin{equation} \label{eq_fracefinite}
	\frace_{(\xi,\eta)}(t):= \int_{(0,t]} e^{-\xi_{s-}} d\eta_s, \quad 0<t<\infty, \end{equation}
	and - in analogy to the Lévy case - refer to its limit 
	\begin{equation} \label{eq_fracedef} \frace_{(\xi,\eta)}^\infty: = \int_{(0,\infty)} e^{-\xi_{s-}} d\eta_s \end{equation}
	as \emph{exponential functional}, whenever it exists.
	
	We prove necessary and sufficient conditions for convergence of $\frace_{\xi,\eta}(t)$ as $t\to \infty$. As it will turn out, by contrast with the classical L\'evy setting here we have to distinguish between almost sure convergence and convergence in probability. We also provide an example of an integral that converges in probability but not almost surely. Another somewhat surprising contrast to the classical setting is the fact that $\lim_{t\to \infty} \xi_t = \infty$ a.s. is no longer necessary for convergence of the integral. Thirdly, the possible degenerate behaviour of $\frace_{\xi,\eta}(t)$ allows for much more flexibility compared to the L\'evy setting. 
		
	Note that exponential functionals of (Markov) additive processes have recently attracted the attention of other researchers as well. Finiteness and tails of the functional \eqref{eq_fracedef} with $\eta_t=t$ are treated in the recent manuscript \cite{alili}. In \cite{salminen1, salminen2} the functional \eqref{eq_fracedef} with  $\eta_t=t$ is studied with an emphasis on moments, while \cite{stephenson} considers similar questions and relations to discrete random structures for the special case of \eqref{eq_fracedef} for $\eta_t=t$ and $\xi$ being a Markov modulated subordinator. Further, exponential integrals of Markov additive processes \eqref{eq_fracefinite} with $\eta_t=t$ appear in the Lamperti-Kiu representation of real self-similar Markov processes, see e.g. \cite{chaumont, kyprianoudereich, kyprianousatit} and references therein. 
	
	The paper is organized as follows. Section \ref{S1}  briefly reviews known results on convergence of exponential integrals of (bivariate) Lévy processes and on perpetuities in a Markovian environment. The class of bivariate Markov additive processes that we shall work with together with some relevant properties are introduced in Section \ref{S1b}. In Section \ref{S3} we present and prove our main result giving necessary and sufficient conditions for convergence of the exponential integral \eqref{eq_fracedef} and discuss the degenerate cases. Finally, Section \ref{S4} is devoted to deriving sufficient conditions for convergence of \eqref{eq_fracedef} that are easier applicable than those from Section \ref{S3}.

	\section{Preliminaries}\label{S1}
	\setcounter{equation}{0}
	
	\subsection{Exponential functionals of L\'evy processes}\label{S1a}
	
	Let $(\xi_t,\eta_t)_{t\geq 0}$ be a bivariate L\'evy process and denote by $(\gamma_\xi,\sigma_\xi^2,\nu_\xi)$ and $(\gamma_\eta,\sigma_\eta^2,\nu_\eta)$ the characteristic triplets of the two marginal processes. We refer to \cite{sato2nd} for any relevant background on L\'evy processes. 
	
	As mentioned in the introduction, Erickson and Maller showed in \cite[Thm. 2]{ericksonmaller05} that the exponential functional of a bivariate L\'evy process \eqref{eq:expfunclevy} exists as a.s. limit as $t\to \infty$ of $\int_0^t e^{-\xi_{s-}}d\eta_s$ if and only if 
	\begin{equation}\label{eq_EM1}
	\lim_{t\to \infty} \xi_t= \infty \; \text{a.s.  and} \quad I_{\xi,\eta}=\int_{(e^a,\infty)} \left(\frac{\log y}{A_\xi(\log y)} \right) |d\bar{\nu}_\eta(y)|<\infty,
	\end{equation}
	where 
	\begin{equation}\label{eq_EM4} A_\xi(x)=\gamma_\xi + \bar{\nu}_\xi^+(1) + \int_{(1,x)} \bar{\nu}_\xi^+(y) dy,\end{equation}
	with
	$$\bar{\nu}_\xi^+(x)=\nu_\xi((x,\infty)), \quad \bar{\nu}_\xi^-(x)=\nu_\xi((-\infty,-x)),\quad   \bar{\nu}_\xi(x)=\bar{\nu}_\xi^+(x)+\bar{\nu}_\xi^-(x),$$
	and $\bar{\nu}_\eta^+$, $\bar{\nu}_\eta^-$ and $\bar{\nu}_\eta$ defined likewise. Hereby $a>0$ is chosen such that $A_\xi(x)>0$ for all $x>a$ and its existence is guaranteed whenever $\lim_{t\to \infty} \xi_t= \infty$ a.s.\\
	Further it is shown in \cite[Thm. 2]{ericksonmaller05} that if 
	$\lim_{t\to \infty} \xi_t= \infty$ a.s. but $I_{\xi,\eta}=\infty$, then 
	\begin{equation}\label{eq_EM2}
	\left|\int_{(0,t]} e^{-\xi_{s-}}d\eta_s \right|\overset{\PP}\longrightarrow \infty, \quad \text{as }t\to\infty,
	\end{equation}
	while for $\lim_{t\to \infty} \xi_t =-\infty$ a.s. or oscillating $\xi$ either \eqref{eq_EM2} holds, or there exists some $k\in \RR\setminus \{0\}$ such that
	\begin{equation}\label{eq_EM3}
	\int_{(0,t]} e^{-\xi_{s-}} d\eta_s  = k(1-e^{-\xi_t})\quad \text{for all }t>0 \quad \text{a.s.}
	\end{equation}
	Note that for any $h>0$ the exponential functional \eqref{eq:expfunclevy} can be discretized in the sense that for all $n\in\NN$
	$$\int_{(0,nh]} e^{-\xi_{s-}}d\eta_s = \sum_{i=0}^{n-1} \int_{(ih,(i+1)h]} e^{-\xi_{s-}}d\eta_s = \sum_{i=0}^{n-1} \left( \prod_{j=0}^{i-1} e^{-(\xi_{(j+1)h} - \xi_{jh} )} \right) \int_{(ih,(i+1)h]} e^{-(\xi_{s-}-\xi_{ih})} d\eta_s,$$
	and hence convergence of the integral is strongly connected to the convergence of discrete-time perpetuities as studied in \cite{goldiemaller00}. Indeed, the proof of the above results given in \cite{ericksonmaller05} relies heavily on choosing an appropriate discretization of \eqref{eq:expfunclevy} and afterwards applying results from \cite{goldiemaller00}.

	\subsection{Markov modulated perpetuities}
	
	Recently Alsmeyer and Buckmann \cite{alsmeyerbuckmann1} generalized the results from \cite{goldiemaller00} to a Markovian environment. More precisely they study convergence of 
	\begin{equation} \label{eq_AB1}
	Z_n:=\sum_{i=1}^n \left( \prod_{j=1}^{i-1} A_j \right) B_i
	\end{equation}
	as $n\to \infty$, where $(A_n,B_n)_{n\in\NN}$ is a sequence of random vectors in $\RR^2$ which is modulated by an ergodic Markov chain $(M_n)_{n\in\NN_0}$ with countable state space $\cS$ and stationary law $\pi$ in the sense that
	conditionally on $M_n=j_n\in \cS$, $n=0,1,2,\ldots $ the random vectors $(A_1,B_1),(A_2,B_2),\ldots$ are independent, and
for all $n\in \NN$ the conditional law of $(A_n,B_n)$ is temporally homogeneous and depends only on $(j_{n-1},j_n)\in \cS^2$.\\
	We write $\PP_j:=\PP(\cdot|M_0=j)$ and $\PP_\pi=\sum_{j\in\cS} \pi_j \PP_j$. 
	Then, under the non-degeneracy conditions 
	\begin{equation}\label{eq_AB2}
	\PP_\pi(A= 0)=0 \text{ and } \PP_\pi (B=0)<1
	\end{equation} 
	for a generic copy $(A,B)$ of the $(A_n,B_n)$ under $\PP_\pi$,
	it follows from \cite[Thm. 3.1]{alsmeyerbuckmann1} that \eqref{eq_AB1} converges a.s. as $n\to\infty$ to a proper random variable given by $Z_\infty$ if and only if
	\begin{align} \label{eq_AB3}
	&\lim_{n\to \infty} \prod_{k=1}^{\tau_n(i)} A_k =0 \;\; \PP_\pi\text{-a.s. and } \\ & \int_{(1,\infty)} \frac{\log q}{\int_{(0,\log q)} \PP_j(-\log |\prod_{\ell=1}^{\tau_1(j)} A_\ell | >x) dx}  \PP_j(W_j\in dq) <\infty \text{ for all }j \in \cS, \nonumber
	\end{align}
	where 
	$$\tau_0(j):=0, \quad \tau_n(j)=\inf\{k>\tau_{n-1}(j): M_k=j\}, \quad j\in \cS$$
	are the return times of $(M_n)_{n\in\NN_0}$, 
	and
	\begin{equation} \label{eq_AB5}
	W_j:=\max_{1\leq k\leq \tau_1(j)} \big|\prod_{\ell=1}^{k-1} A_\ell B_k\big|, \quad j\in\cS.
	\end{equation}
	Note that $\lim_{n\to \infty} \prod_{k=1}^{\tau_n(j)} A_k =0$ $\PP_\pi$-a.s. necessarily implies that $\PP_j(| \prod_{\ell=1}^{\tau_1(j)} A_\ell| \geq 1 )<1$ and hence the denominator in the integral in \eqref{eq_AB3} is non-zero.\\
	We remark that in \cite[Thm. 3.1 and Rem. 3.3]{alsmeyerbuckmann1} the authors state that \eqref{eq_AB3} for some $j\in \cS$ is equivalent to \eqref{eq_AB3} for all $j\in \cS$. As we were not able to follow their argumentation in the discrete time setting or to derive a similar result in the continuous time setting, we stick to the stronger assumption here.\\
	Interestingly, by contrast with the case of i.i.d. sequences $(A_n,B_n)_{n\in\NN}$, if almost sure convergence fails, the perpetuity \eqref{eq_AB1} can still converge in probability. More precisely, from \cite[Thm. 3.4]{alsmeyerbuckmann1} we derive under the assumption \eqref{eq_AB2} that $\PP_j(Z_n\in \cdot )$ for some $j\in \cS$ converges weakly to some probability measure $Q_j$ as $n\to\infty$  if  
	\begin{align}\label{eq_AB7}
	&\lim_{n\to \infty} \prod_{k=1}^{\tau_n(j)} A_k =0 \;\PP_\pi \text{-a.s. and } \\ & \int_{(1,\infty)} \frac{\log q}{\int_{(0,\log q)} \PP_j(-\log |\prod_{\ell=1}^{\tau_1(j)} A_\ell | >x) dx}  \PP_j(|Z_{\tau_1(j)}| \in dq) <\infty. \nonumber
	\end{align}
	In this case, there exists a random variable $Z_\infty$ such that $Q_j(\cdot)=\PP_j(Z_\infty\in \cdot)$ and $Z_n\overset{\PP_j}\rightarrow Z_\infty$. Moreover, if \eqref{eq_AB7} holds for some $j\in\cS$, convergence in probability is valid for all $j\in\cS$.\\
	Furthermore, if \eqref{eq_AB7}
	is violated and the degeneracy condition 
	\begin{equation}\label{eq_AB9}
	\PP_\pi(A_1c_{M_1} + B_1= c_{M_0})=1 \text{ for suitable constants }c_j\in\RR, j\in\cS,
	\end{equation}
	fails, then 
	\begin{equation}\label{eq_AB10}
	|Z_n|\overset{\PP_\pi}\longrightarrow \infty, \quad n\to\infty.
	\end{equation}

	\section{Bivariate Markov additive processes}\label{S1b}
	\setcounter{equation}{0}

	The theory of Markov additive processes (MAPs) goes back to \c{C}inlar \cite{cinlarMAP1, cinlarMAP2} and has been enhanced since then by various researchers (see e.g. \cite{asmussenkella, kyprianoudereich, grigelionis, klusikpalm}). In this paper we restrict to the most popular framework of MAPs, that is to Markov modulated Lévy processes similar to the setting in \cite{asmussenkella} and \cite{klusikpalm}. We refer to \cite{asmussen} for a standard modern treatment of the topic and set notation as follows. 
	
	Let $(J_t)_{t\geq 0}$ be a right-continuous, ergodic, continuous time Markov chain with finite or countable state space $\cS\subseteq \NN$, intensity matrix $Q=(q_{i,j})_{i,j\in\cS}$ and stationary law $\pi=(\pi_j)_{j\in\cS}$. We denote the jump times of $(J_t)_{t\geq 0}$ by $\{T_n, n\in\NN_0\},$ with $T_0:=0$, while
	\begin{equation}\label{eq_MAP03}
	\tau_0(j):=0, \quad \text{and}\quad \tau_n(j):= \inf\{T_k>\tau_{n-1}(j): J_{T_k}=j\}, \; n\in\NN, j\in\cS,
	\end{equation}
	are its return times and
	\begin{equation}
	\tau_0^-(j)=0, \quad \text{and}\quad \tau_n^-(j):= \inf\{T_k>\tau_{n-1}(j): J_{T_k}\neq j\}, \; n\in\NN, j\in\cS,
	\end{equation}
	are the corresponding exit times under $\PP_j$. The sojourn time of $(J_t)_{t\geq 0}$ in a state $j\in\cS$ is denoted as
	$$\TT_j:=\{t\geq 0: J_t=t\}$$
	and clearly under $\PP_j$ we have $\TT_j=\bigcup_{n\in\NN} [\tau_{n-1}(j),\tau_n^-(j))$.
	
	Further, let $(\xi^{(j)}_t,\eta^{(j)}_t)_{t\geq 0}$, $j\in \cS,$ be bivariate L\'evy processes with characteristic triplets $(\gamma^{(j)}, \Sigma^{(j)}, \nu^{(j)})$ where $\gamma^{(j)}=(\gamma_{\xi^{(j)}}, \gamma_{\eta^{(j)}})$, $\nu^{(j)}$ denotes the L\'evy measure and 
	$$\Sigma^{(j)}= \Big(\begin{smallmatrix}
	\sigma_{\xi^{(j)}}^2 & \sigma_{\xi^{(j)},\eta^{(j)}}^2 \\\sigma_{\xi^{(j)},\eta^{(j)}}^2 & \sigma_{\eta^{(j)}}^2
	\end{smallmatrix}\Big).$$
	Set 
	\begin{equation}\label{eq_MAP01}
	(\xi_t,\eta_t):= (X_t^{(1)},Y_t^{(1)})+  (X_t^{(2)},Y_t^{(2)}), \quad t\geq 0,
	\end{equation}
	where
	$(X_t^{(1)},Y_t^{(1)})$ behaves in law like $(\xi^{(j)}_t,\eta^{(j)}_t)$ whenever $J_t=j$, while $(X_t^{(2)},Y_t^{(2)})$ is a pure jump process given by
	\begin{equation}\label{eq_MAP02}
	(X_t^{(2)},Y_t^{(2)}) = \sum_{n\geq 1} \sum_{i,j\in \cS} Z_n^{(i,j)} \mathds{1}_{\{J_{T_n-}=i, J_{T_n}=j, T_n\leq t \}},
	\end{equation}
	for i.i.d. random variables $(Z_n^{(i,j)})_{n\in \NN}$ in $\RR^2$ with distribution functions $F^{(i,j)}$, $i,j\in\cS$ (possibly with all mass/an atom in $0$). As starting value we use $(\xi_0,\eta_0)=(0,0)$ and throughout we assume that neither $\xi$ nor $\eta$ is degenerate constantly equal to $0$ a.s. 
	
	The joint process $(\xi_t,\eta_t, J_t)_{t\geq 0}$ is a MAP and we refer to $(J_t)_{t\geq 0}$ as its \emph{Markovian component}, while $(\xi_t,\eta_t)_{t\geq 0}$ is its \emph{additive component}. Clearly the marginal processes $(\xi_t, J_t)_{t\geq 0}$ and $(\eta_t, J_t)_{t\geq 0}$ are MAPs as well.
	
	We assume that the introduced processes are defined on a complete filtered probability space $(\Omega,\cF, \FF=(\cF_t)_{t\geq 0}, \PP)$ where $\FF$ is the augmented natural filtration induced by $(\xi_t,\eta_t, J_t)_{t\geq 0}$. We write $\PP_j:= \PP(\cdot|J_0=j)$ and $\PP_\pi=\sum_{j\in\cS} \pi_j \PP_j$, with the corresponding expectations $\EE_j$ and $\EE_\pi$ defined accordingly.
	
	Note that for simplicity we will sometimes abuse notation and - given $J_t=j$ - identify the processes $X_t^{(1)}$ and $\xi_t^{(j)}$ or  $Y_t^{(1)}$ and $\eta_t^{(j)}$ where this is suitable.
	
	Due to the switching L\'evy process character of the first summand in the additive component it is not surprising that, given the Markovian component, these components admit a L\'evy-It\^o-type decomposition (see e.g. \cite[Thm. 19.3]{sato2nd}). Exemplarily we decompose $(\eta_t)_{t\geq 0}$ as
	\begin{align} 
	\eta_t & =  \int_{(0,t]} \gamma_{\eta^{(J_s)}} ds + \int_{(0,t]} \sigma_{\eta^{(J_s)}}^2 dW_s +  \int_{(0,t]} \int_{|x|\geq 1} x N_{\eta^{(J_s)}}(ds, dx) +  \nonumber\\
	 & \quad + \lim_{\varepsilon \to 0} \int_{(0,t]} \int_{\varepsilon \leq |x|< 1} x (N_{\eta^{(J_s)}}(ds, dx) - ds\,\nu_{\eta^{(J_s)}}(dx)) +  Y_t^{(2)} \nonumber \\
	&=: \gamma^\eta_t + W^\eta_t + Y^{b,\eta}_t + Y^{s,\eta}_t  +  Y_t^{(2)},\label{eq_MAP04}
	\end{align}
	where $(W_t)_{t\geq 0}$ is a standard Brownian motion and $N_{\eta^{(j)}}$ are Poisson random measures with intensity measures $ds\,\nu_{\eta^{(j)}}(dx)$, respectively. Using this decomposition it is straightforward to define integration with respect to the additive component of a MAP given its Markovian component. \\
	Another property of the additive components that carries over from L\'evy processes and which will be of importance in our results is the well-known fact that L\'evy processes in $\RR$ either drift to $\pm \infty$ or oscillate. To formulate the analoguous result for MAPs, we introduce their \emph{long-term mean} (here for the MAP $(\xi_t,J_t)_{t\geq 0}$) 
	\begin{equation}\label{eq-MAP06}
	\kappa_\xi:= \sum_{j\in\cS} \pi_j \left( \gamma_{\xi^{(j)}} + \int_{|x|\geq 1} x \nu_{\xi^{(j)}}(dx)  \right) + \sum_{\substack{(i,j)\in \cS\times \cS \\ i\neq j}} \pi_i  q_{i,j} \int_\RR x F_\xi^{(i,j)}(dx),
	\end{equation}
	which is finite whenever $\EE[|\xi^{(j)}_1|]<\infty$ and $\int_\RR |x| F_\xi^{(i,j)}(dx)<\infty$ for all $i,j\in\cS$. Whenever $\cS$ is finite, $\kappa_\xi$ fully determines the long-term behaviour of $(\xi_t)_{t\geq 0}$ as follows (see \cite[Prop. XI.2.10]{asmussen}):
	\begin{align}
	\kappa_\xi>0 \quad &\Rightarrow  \quad \lim_{t\to \infty} \xi_t =\infty \;\; \PP_\pi\text{-a.s.}, \label{eq-MAP07}\\
	\kappa_\xi<0 \quad &\Rightarrow  \quad \lim_{t\to \infty} \xi_t =-\infty \;\; \PP_\pi\text{-a.s.}, \quad \text{while} \label{eq-MAP08}\\
	\kappa_\xi=0 \text{ and } \PP_\pi\big(\sup_{t\geq 0}|\xi_t|<\infty\big)<1  \quad &\Rightarrow  \quad \limsup_{t\to \infty} \xi_t =\infty \text{ and } \liminf_{t\to \infty} \xi_t =-\infty \;\; \PP_\pi\text{-a.s.} \label{eq-MAP09}
	\end{align}
	For countable $\cS$, as noted in \cite[Cor. 2.2]{kellarama}, $0<\kappa_\xi<\infty$ implies $\lim_{t\to \infty} \xi_t/t =\kappa_\xi \;\; \PP_\pi\text{-a.s.}$ such that in particular $\lim_{t\to \infty} \xi_t =\infty \;\; \PP_\pi\text{-a.s.}$

	\section{Main results and Discussion}\label{S3}
	\setcounter{equation}{0}
	
	\subsection{Main  theorem}
	
	Recall from \eqref{eq_fracedef} that given a bivariate Markov additive process $(\xi_t,\eta_t,J_t)_{t\geq 0}$ with Markovian component $(J_t)_{t\geq 0}$ as introduced above, we denote the exponential integral of $(\xi_t,\eta_t)_{t\geq 0}$ as
	\begin{equation*}
	 \frace_{(\xi,\eta)}(t):= \int_{(0,t]} e^{-\xi_{s-}} d\eta_s, \quad 0<t<\infty. \end{equation*}
	
	The following theorem provides necessary and sufficient conditions for almost sure and weak convergence of $\frace_{(\xi,\eta)}(t)$ as $t\to \infty.$ To formulate the conditions, we set
	$$A^j_\xi(x):=A_{\xi^{(j)}}(x)-q_{j,j} \Big(\PP_j\big(\xi_{\tau_1(j)}-\xi_{\tau^-_1(j)}\in (1,\infty)\big) + \int_{(1,x)} \PP_j\big(\xi_{\tau_1(j)}-\xi_{\tau^-_1(j)}\in (y,\infty)\big) dy \Big),$$
	with $A_{\xi^{(j)}}$ from \eqref{eq_EM4}, and
	$$\bar{\nu}_\eta^j(dy):= \nu_{\eta^{(j)}}(dy) - q_{j,j}\Big( \PP_j(\eta_{\tau_1(j)}-\eta_{\tau^-_1(j)}\in dy) + \PP_j\Big( \int_{[\tau_1^-(j),\tau_1(j)]} e^{-(\xi_{s-}-\xi_{\tau_1^-(j)})} d\eta_s \in dy\Big)\Big).$$
		
	\begin{theorem}\label{thmnessandsuff}
	Assume $\lim_{t\in \TT_j, t\to \infty} \xi_t=\infty$ $\PP_j$-a.s. for some hence all $j\in \cS$. Then $\frace_{(\xi,\eta)}(t)\to \frace_{(\xi,\eta)}^\infty$ $\PP_\pi$-a.s. as $t\to \infty$ for some random variable $\frace_{(\xi,\eta)}^\infty$ if and only if  for all $j\in \cS$
	\begin{equation}\label{eq-nesssuffas}
	\int_{(1,\infty)} \frac{\log q}{\int_{(0,\log q]} \PP_j(\xi_{\tau_1(j)}>u) du } \PP_j\left( \sup_{0< t \leq \tau_1(j)} \Big|\int_{(0,t]} e^{-\xi_{s-}} d\eta_s  \Big|\in dq \right)<\infty.
	\end{equation}
	Further, if $\lim_{t\in \TT_j, t\to \infty} \xi_t=\infty$ $\PP_j$-a.s. and 
	\begin{equation}\label{eq-nesssuffweak}
	\int_{(1,\infty)} \frac{\log q}{A^j_\xi(\log q)} |d\bar{\nu}_\eta^j(q)| <\infty,
	\end{equation}
	for some $j\in \cS$, then  
	$\frace_{(\xi,\eta)}(t)\to \frace_{(\xi,\eta)}^\infty$ in $\PP_j$-probability as $t\to\infty$ for all $j\in\cS$.\\
	Conversely, if $\liminf_{t\in \TT_j, t\to \infty}\xi_t<\infty$ $\PP_j$-a.s. for some $j\in \cS$ or if \eqref{eq-nesssuffweak} fails for all $j\in\cS$, then either there exists a (unique) sequence $\{c_j,j\in\cS\}$ in $\RR$ such that the functional is \emph{degenerate} in the sense that
	\begin{equation}\label{eq_deg00}
	\frace_{(\xi,\eta)}(t) = \int_{(0,t]} e^{-\xi_{s-}} d\eta_s = c_{J_0}- c_{J_t} e^{-\xi_{t}} \quad \PP_{\pi}\text{-a.s.} 
	\end{equation}
	for all $t\geq 0$, or
	$$|\frace_{(\xi,\eta)}(t)|\overset{\PP_\pi}\longrightarrow \infty \quad \text{as }t\to\infty.$$		
	\end{theorem}

The proof of this theorem is given in Sections \ref{s42}, \ref{s43} and \ref{s44} below, which are devoted to almost sure convergence, convergence in probability and divergence, respectively. In particular Lemma \ref{lem-help2} proves equivalence of $\lim_{t\in \TT_j, t\to \infty} \xi_t=\infty$ $\PP_j$-a.s. for some and for all $j\in \cS$. 

 \begin{remark}
 If $\cS$ is finite and \eqref{eq_deg00} is not valid, but $\frace_{(\xi,\eta)}(t)$ converges in $\PP_j$-probability, then $\frace_{(\xi,\eta)}(t)$ also converges $\PP_\pi$-a.s., i.e. the two types of convergence are equivalent in this case. This is also true in the discrete setting as argued in \cite[Rem. 3.8]{alsmeyerbuckmann1} and the argumentation given there carries over to the continuous-time setting studied here:
 By the above theorem, convergence in probability implies $\lim_{t\in \TT_j, t\to \infty} \xi_t=\infty$ a.s. and \eqref{eq-nesssuffweak}. From the proof of the convergence in probability part of Theorem \ref{thmnessandsuff} (Proposition \ref{prop-weak} below) we will see, that this implies $\mathbb{P}_j$-a.s. convergence of the ``conflated exponential integrals'' $\hat{\frace}_{(\xi,\eta)}(t)$ for all $j\in S$ as they are defined in that proof. But for these one easily verifies 
\begin{align*}
 \lim_{t\to\infty}	|\frace_{(\xi,\eta)}(t)-\frace_{(\xi,\eta)}^\infty|&\leq \lim_{t\to\infty} \max_{j\in S}|\hat{\frace}_{(\xi,\eta)}^j(t)-\frace_{(\xi,\eta)}^\infty| =0 \quad  \PP_\pi\text{-a.s.}
 \end{align*}
  because $\cS$ is finite.
 \end{remark}
 
 \subsection{Examples}
	
	Note that by contrast with the standard L\'evy case, $\xi_t\to \infty$ is not necessary for almost sure convergence of the exponential integral. This is further outlined by the following example.
	
	\begin{example}\label{ex-xinottoinfty}
		Let $\cS=\NN$ and let $(J_t)_{t\geq 0}$ be a continuous time petal flower Markov chain (see e.g. \cite{alsmeyerbuckmann2}) with intensity matrix  
		$$Q=(q_{i,j})_{i,j\in\NN} = \begin{pmatrix}
		-q & q_{1,2} & q_{1,3} & \ldots \\
		q & -q & 0 & \ldots \\
		q & 0 & -q & \\
		\vdots & \vdots & & \ddots
		\end{pmatrix}$$
		for some $q>0$ fixed and $q_{1,j}=q p_{1,j}$ $j=2,3,\ldots$ for transition probabilities $p_{1,j}>0$, $j\in\NN\backslash\{1\}.$ 
		Then $(J_t)_{t\geq 0}$ is an irreducible and recurrent Markov process with stationary distribution
		$$\pi_1=\frac12, \quad \text{and}\quad  \pi_j=\frac{p_{1,j}}{2}=\frac{q_{1,j}}{2q}, \; j=2,3,\ldots$$
		As additive component we choose $\xi$ and $\eta$ to be conditionally independent with $Y_t^{(2)}\equiv 0$, that is the second component of $(\xi_t,\eta_t)_{t\geq 0}$ has no common jumps with $(J_t)_{t\geq 0}$. Further
			$$\xi_t:=X_t^{(2)}: = \sum_{n\geq 1} \sum_{i,j\in\NN} Z_n^{(i,j)} \mathds{1}_{\{J_{T_n-}=i,J_{T_n}=j, T_n\leq t\}}$$
			where 
			$$Z_n^{(i,j)}:=Z^{(i,j)}:=\begin{cases}
			-p_{1,j}^{-1}, & i=1,\\
			2+p_{1,i}^{-1}, & j=1,\\
			0, & \text{otherwise},
			\end{cases}$$
			such that $\EE_1[Z^{(1,J_{T_1})}]=-\infty.$
		 	We then directly observe that 
			$$\xi_{\tau_n(1)} = 2n \to \infty \quad \PP_1\text{-a.s.}$$
			Nevertheless $(\xi_t)_{t\geq 0}$ does not tend to $\infty$ as $t\to \infty$. Indeed, as
			$$\xi_{T_n}=\begin{cases}
			n,& n\text{ even},\\
			n-1 -p_{1,J_{T_n}}^{-1}, & n \text{ odd},
			\end{cases} \quad \text{under }\PP_1,$$
			and as $(\xi_t)_{t\geq 0}$ is constant between two jumps of $(J_t)_{t\geq 0}$, setting $N_t:=\sum_{n\geq 1} \mathds{1}_{\{T_n\leq t\}}$ we clearly obtain
			$$\limsup_{t\to\infty}\frac{\xi_{t}}{t}=\limsup_{t\to\infty}\frac{\xi_{T_{N_t}}}{t} =\lim_{t\to\infty} \frac{N_t}{t} = q \quad \PP_1\text{-a.s}. $$
			This is consistent with the (only formal!) computation of $\kappa_\xi$ which would yield
			\begin{align*}
			\sum_{\substack{(i,j)\in \NN\times \NN \\ i\neq j}} \pi_i  q_{i,j} \EE[Z^{(i,j)}] &= \sum_{i\in\NN} \pi_i \sum_{j\in\NN, j\neq i} q_{i,j} Z^{(i,j)} \\
			&= \pi_1 \sum_{j=2}^\infty q p_{1,j} \left(-p_{1,j}^{-1}\right) + \sum_{i=2}^\infty \pi_i q_{i,1} (2+p_{1,i}^{-1})= q.
			\end{align*}
			On the other hand 
			$$\EE_1[\xi_{T_{2n+1}}]= 2n+\EE_1[\xi_{T_1}]= 2n+\EE_1[Z^{(1,J_{T_1})}]=-\infty, \quad n\in\NN,$$
			which implies for any $x>0$
			$$\sum_{n\geq 0} \PP_1\left(-\xi_{T_{2n+1}} >x \right) = \sum_{n\geq 0} \PP_1\left( -\xi_{T_1} > x+2n \right) =\infty$$
			such that by the Borel-Cantelli lemma we conclude
			$$\liminf_{t\to\infty}\xi_{t}=-\infty \quad \PP_1\text{-a.s.}.$$
			Thus $\xi$ is oscillating. \\	
			Still, choosing $\eta_t=\int_{(0,t]} \gamma_{\eta^{(J_s)}} ds$ with 
			$$\gamma_{\eta^{(j)}}=\begin{cases}
			1,& j=1,\\
			0,& \text{otherwise},
			\end{cases}$$
			we observe that under $\PP_1$ the exponential integral
			\begin{align*}
			\int_{(0,t]} e^{-\xi_{s-}}d\eta_s &= 
			\int_{(0,t]} e^{-\xi_{s-}} \gamma_{\eta^{(J_s)}} ds 
			= \int_{(0,t]} e^{-N_{s-}} \mathds{1}_{\{J_s=1\}} ds
			\end{align*}
			converges $\PP_1$-a.s. as $t\to\infty$.
	\end{example}

	The following example provides a scenario where the exponential integral converges in probability but not almost surely.
	
	\begin{example}\label{ex-weaknotas}
	Set $\cS=\NN_0=\NN\cup \{0\}$ and let $(J_t)_{t\geq 0}$ be a continuous time Markov chain which behaves like the petal flower chain described in Example \ref{ex-xinottoinfty} on $\NN$, but has an additional special state $0$ connected solely to state $2$. More precisely, we set
	$$Q=(q_{i,j})_{i,j\in\NN_0} = \begin{pmatrix}
 -q &0 &q &0 &\ldots \\
 0 & -q & q_{1,2} & q_{1,3} & \ldots \\
 q/2 &	q/2 & -q & 0 & \ldots \\
 0&	q & 0 & -q & \\
\vdots&	\vdots & \vdots & & \ddots
	\end{pmatrix}$$
	for some $q>0$ fixed and $q_{1,j}=q p_{1,j}$ $j=2,3,\ldots$ for transition probabilities $p_{1,j}>0$, $j\in\NN\backslash\{1\}.$ Then $(J_t)_{t\geq 0}$ is an irreducible and recurrent Markov chain.\\
	Further we assume the bivariate L\'evy process $(\xi^{(0)}_t,\eta^{(0)}_t)_{t\geq 0}$ to be such that the exponential integral \eqref{eq:expfunclevy} of $(\xi^{(0)}_t,\eta^{(0)}_t)_{t\geq 0}$ converges a.s., i.e. such that \eqref{eq_EM1} is fulfilled. Apart we set $(\xi^{(j)}_t,\eta^{(j)}_t)\equiv(0,0)$ for all $j\in\NN$. Additionally $X_t^{(2)}\equiv 0$, that is the first component of $(\xi,\eta)$ has no common jumps with $(J_t)_{t\geq 0}$. As second component we assume 
	$$Y_t^{(2)}=\sum_{n\geq 1} \sum_{i,j\in \cS} Z^{(i,j)} \mathds{1}_{\{J_{T_n-}=i, J_{T_n}=j, T_n\leq t \}}$$
	with $$Z^{(0,2)}=Z^{(2,0)}=0, \quad \text{and} \quad Z^{(1,j)}=-Z^{(j,1)}=e^{1/p_{1,j}} \; \text{a.s.}$$
	Then $A_\xi^0(x)=A_{\xi^{(0)}}(x)$ and $\nu_\eta^0(dy) = \nu_{\eta^{(0)}}(dy)$ and hence \eqref{eq-nesssuffweak} is fulfilled by assumption for $j=0$. Hence the exponential integral converges in $\PP_j$-probability. Nevertheless, almost sure convergence is impossible as the integral oscillates with 
	$$-\infty=\liminf_{t\to \infty} \frace_{(\xi,\eta)}(t) < \limsup_{t\to \infty} \frace_{(\xi,\eta)}(t) =\infty \quad \text{a.s.}$$
	as can be shown again using the Borel-Cantelli lemma.
	\end{example}

	\subsection{Degeneracy of $\frace_{(\xi,\eta)}(t)$}\label{s41}
	
	Before we prove Theorem \ref{thmnessandsuff} we will discuss the possible degenerate behaviour of $\frace_{(\xi,\eta)}$ in more detail. This study of degeneracy will rely on a combination of results from \cite{alsmeyerbuckmann1} and \cite{ericksonmaller05}.
	
	 Recall first that degeneracy in the classical L\'evy case $\cS=\{1\}$ is characterized by \eqref{eq_EM3}. To study degeneracy for larger state spaces $\cS$, note that at jump times of $(J_t)_{t\geq 0}$ we can rewrite $\frace_{(\xi,\eta)}$ as
	\begin{align}
	\frace_{(\xi,\eta)}(T_n)
	&= \sum_{i=1}^{n} \left(\prod_{k=1}^{i-1} e^{-(\xi_{T_k}-\xi_{T_{k-1}})}\right) \int_{(T_{i-1},T_i]} e^{-(\xi_{s-}-\xi_{T_{i-1}})} d\eta_s\nonumber\\
	&=:\sum_{i=1}^n \left(\prod_{k=1}^{i-1} A_k \right) B_i, \label{eq_deg01}
	\end{align}
	where 
	$$(A_n,B_n)_{n\in\NN} = \left(e^{-(\xi_{T_n}-\xi_{T_{n-1}})} , \int_{(T_{n-1},T_n]} e^{-(\xi_{s-}-\xi_{T_{n-1}})} d\eta_s  \right)_{n\in\NN}$$
	is a sequence of random vectors modulated by a Markov chain $(M_n)_{n\in \NN}$ which is the discrete time jump chain of $(J_t)_{t\geq 0}$. W.l.o.g. we assume that $(M_n)_{n\in \NN}$  inherits the ergodicity from $(J_t)_{t\geq 0}$ (see the proof of Prop. \ref{prop_frace2conv} for more details). Then its stationary law $\pi_M$ is equivalent to $\pi$ and as shown in \cite[Eq. (17) and Lemma 4.1]{alsmeyerbuckmann1} degeneracy of the Markov modulated perpetuity \eqref{eq_deg01} in the sense of \eqref{eq_AB9} is equivalent to the existence of a unique sequence $\{c_j,j\in\cS\}$ in $\RR$ such that 
	\begin{equation}\label{eq_deg02}
	\PP_j(A_j^1 c_j + B_j^1=c_j)=1 \quad \text{for all/some}\quad j\in\cS,
	\end{equation}
	where in our setting
	$$A_j^1= e^{-\xi_{\tau_1(j)}},\quad \text{and} \quad B_j^1= \int_{(0,\tau_1(j)]} e^{-\xi_{s-}} d\eta_s.$$
	Further, by \cite[Prop. 4.6]{alsmeyerbuckmann1} validity of \eqref{eq_deg02} implies (and is thus equivalent to)
	\begin{equation}\label{eq_deg03} 
	\frace_{(\xi,\eta)}(T_n) = \int_{(0,T_n]} e^{-\xi_{s-}} d\eta_s = c_{M_0}- c_{M_n} e^{-\xi_{T_n}} \quad \PP_{\pi_M}\text{-a.s.} 
	\end{equation}
	for any $n\in\NN$, which in turn is equivalent to \eqref{eq_deg00} 
	for all $t\geq 0$ as will be shown in the following proposition.

	\begin{proposition}\label{prop_degchar}
		Assume there exists a sequence 	$\{c_j,j\in\cS\}$ in $\RR$ such that \eqref{eq_deg03} holds for any $n\in\NN$. Then \eqref{eq_deg00} holds for all $t\geq 0$.
	\end{proposition}
	\begin{proof}
		We show this by contradiction and assume there exists a sequence $\{c_j,j\in\cS\}$ in $\RR$ such that \eqref{eq_deg03} holds for all $n\in\NN$, but for any sequence $\{c_j,j\in\cS\}$ there is $t'\geq 0$ such that
		\begin{equation}\label{eq_deg06}
		\PP_\pi\left( \int_{(0,t']} e^{-\xi_{s-}} d\eta_s = c_{J_0}- c_{J_{t'}} e^{-\xi_{t'}}\right)<1.
		\end{equation}
		We choose the sequence $\{c_j,j\in\cS\}$ such that \eqref{eq_deg03} holds for all $n$ and then choose $t'$ such that \eqref{eq_deg06} holds for the given $\{c_j,j\in\cS\}$. \\
		We will first show that 
		\begin{equation}\label{eq_deg07}
		\PP_\pi\left( \int_{(0,t']} e^{-\xi_{s-}} d\eta_s = \tilde{c}_{J_0}- \tilde{c}_{J_{t'}} e^{-\xi_{t'}}\right)<1 \quad \text{for all sequences } \{\tilde{c}_j,j\in\cS\}.
		\end{equation}
		To do so, let $T'$ be the first jump time of $J$ after $t'$, i.e. $T':=\inf\{t\geq t': J_t\neq J_{t-}\}$. Then from \eqref{eq_deg03} $\PP_\pi$-a.s.
		\begin{equation}
		\int_{(0,t']} e^{-\xi_{s-}} d\eta_s + \int_{(t',T']} e^{-\xi_{s-}} d\eta_s = \int_{(0,T']} e^{-\xi_{s-}} d\eta_s = c_{J_0}- c_{J_{T'}} e^{-\xi_{T'}}.\label{eq_deg08}
		\end{equation}
		Assume there exists a sequence $\{\tilde{c}_j,j\in\cS\}$ such that
		\begin{equation}\label{eq-deg09} \PP_\pi\left( \int_{(0,t']} e^{-\xi_{s-}} d\eta_s = \tilde{c}_{J_0}- \tilde{c}_{J_{t'}} e^{-\xi_{t'}}\right)=1,\end{equation}
		then from \eqref{eq_deg08} $\PP_\pi$-a.s.
		\begin{align*}
		\tilde{c}_{J_0}- \tilde{c}_{J_{t'}} e^{-\xi_{t'}}  & = c_{J_0}- c_{J_{T'}} e^{-\xi_{T'}}  - \int_{(t',T']} e^{-\xi_{s-}} d\eta_s\\
		\Leftrightarrow \tilde{c}_{J_0} - c_{J_0}  & =    e^{-\xi_{t'}}\left( \tilde{c}_{J_{t'}} - c_{J_{T'}} e^{-(\xi_{T'}- \xi_{t'})}  - \int_{(t',T']} e^{-(\xi_{s-}-\xi_{t'})} d\eta_s \right), 
		\end{align*}
		where the two factors on the right hand side are (conditionally on $(J_t)_{t\geq 0}$) independent, while the left hand side is a constant. Thus we deduce
		$$\tilde{c}_{J_0} - c_{J_0} =0= \tilde{c}_{J_{t'}} - c_{J_{T'}} e^{-(\xi_{T'}- \xi_{t'})}  - \int_{(t',T']} e^{-(\xi_{s-}-\xi_{t'})} d\eta_s \quad \PP_\pi\text{-a.s.}$$
		which implies $\{c_j,j\in\cS\}=\{\tilde{c}_j,j\in\cS\}$ in contradiction to \eqref{eq-deg09}, such that \eqref{eq_deg07} is true. \\
		Finally, to prove the assertion of the proposition note that from \eqref{eq_deg08} we have $\PP_\pi$-a.s.
		\begin{align*}
		\int_{(0,t']} e^{-\xi_{s-}} d\eta_s &= \int_{(0,T']} e^{-\xi_{s-}} d\eta_s  - \int_{(t',T']} e^{-\xi_{s-}} d\eta_s \\
		&= c_{J_0}- c_{J_{T'}} e^{-\xi_{T'}}  - e^{-\xi_{t'}} \int_{(t',T']} e^{-(\xi_{s-}-\xi_{t'})} d\eta_s\\
		&= c_{J_0}-  e^{-\xi_{t'}}\left( c_{J_{T'}} e^{-(\xi_{T'}-\xi_{t'})}  +  \int_{(t',T']} e^{-(\xi_{s-}-\xi_{t'})} d\eta_s \right).
		\end{align*}
		Conversely, from \eqref{eq_deg07} and due to independence
		\begin{align*}
		\PP_\pi \Bigg( \int_{(0,t']} e^{-\xi_{s-}} d\eta_s &= c_{J_0}-  e^{-\xi_{t'}}\underbrace{\left( c_{J_{T'}} e^{-(\xi_{T'}-\xi_{t'})}  +  \int_{(t',T']} e^{-(\xi_{s-}-\xi_{t'})} d\eta_s \right)}_{=:f((\xi_s,\eta_s,J_s)_{t'<s\leq T'})} \Bigg) \\
		&= \EE\Bigg[ \PP_\pi\left( \int_{(0,t']} e^{-\xi_{s-}} d\eta_s = c_{J_0}-  C e^{-\xi_{t'}} \right) \big| f((\xi_s,\eta_s,J_s)_{t'<s\leq T'})=C \Bigg]<1
		\end{align*}
		which yields the desired contradiction.
	\end{proof}
	
	\begin{remark}
		Note that uniqueness of the sequence $\{c_j,j\in\cS\}$ in \eqref{eq_deg03} as shown in \cite{alsmeyerbuckmann1} directly implies uniqueness of the sequence $\{c_j,j\in\cS\}$ in \eqref{eq_deg00}.
	\end{remark}
	
	The following proposition gives a necessary and sufficient condition for a degenerate behaviour of the exponential integral in terms of $(\xi_t,\eta_t,J_t)_{t\geq 0}$.

	\begin{proposition}\label{prop_degcond}
		Assume \eqref{eq_deg00} holds for all $t\geq 0$ and some sequence $\{c_j, j\in\cS\}$. Then 
		\begin{equation}\label{eq_deg13}
		\eta_t= -\int_{(0,t]} c_{J_{s-}} dU_s - \int_{(0,t]} dc_{J_s}, \quad t \geq 0,\; \PP_\pi\text{-a.s.}
		\end{equation}
		where $(U_t)_{t\geq 0}= (\Log(e^{-\xi_t}))_{t\geq 0}$ is the stochastic logarithm of $(e^{-\xi_t})_{t\geq 0}$, i.e. the unique solution of the SDE $dU_t = e^{\xi_{t-}}d e^{-\xi_t}$, $t\geq 0$, $U_0=0$, given by
		\begin{equation}\label{eq_deg14}
		U_t= - \xi_t + \sum_{0<s\leq t} (e^{-\Delta \xi_s}-1+\Delta \xi_s) + \frac{1}{2} \int_{(0,t]} \sigma_{\xi^{(J_s)}}^2 ds, \quad t\geq 0.
		\end{equation}
		Conversely, if \eqref{eq_deg13} holds for some sequence $\{c_j, j\in\cS\}$, then \eqref{eq_deg00} is fulfilled for all $t\geq 0$.
	\end{proposition}
	\begin{proof}
		Recall first that by \cite[Thm. II.37]{protter} for all $j\in\cS$ it holds $e^{-\xi^{(j)}_t}=\cE(U^{(j)})_t$ for some L\'evy processes $U^{(j)}$ defined by 
		\begin{equation} \label{eq_deg10} U^{(j)}_t= - \xi^{(j)}_t + \sum_{0<s\leq t} (e^{-\Delta \xi^{(j)}_s}-1+\Delta \xi^{(j)}_s) + t \frac{\sigma_{\xi^{(j)}}^2}{2}, \quad t\geq 0,\, j\in\cS,\end{equation}
		where $(\cE(U^{(j)})_t)_{t\geq 0}$ denotes the Doleans-Dade stochastic exponential of $(U^{(j)}_t)_{t\geq 0}$, i.e. the unique solution of the SDE
		\begin{equation} \label{eq_deg10b} d\cE(U^{(j)})_t =  \cE(U^{(j)})_{t-} dU^{(j)}_t, \quad \cE(U^{(j)})_0=1.\end{equation}
		Now assume \eqref{eq_deg00} for all $t\geq 0$, then for all $j\in\cS$  we observe immediately $\PP_j$-a.s. for $t<T_1$
		\begin{align*}
		\int_{(0,t]}e^{-\xi^{(j)}_{s-}}d\eta^{(j)}_s 
		&= 	\int_{(0,t]}e^{-\xi_{s-}}d\eta_s 
		= c_{J_0}-c_{J_t}e^{-\xi_t}
		= c_j - c_j e^{-\xi^{(j)}_t} = c_j (1-e^{-\xi^{(j)}_t}),
		\end{align*} 
		that is, in terms of $U^{(j)}$ from \eqref{eq_deg10},
		$$\int_{(0,t]} \cE(U^{(j)})_{s-} d\eta^{(j)}_s  = c_j (1- \cE(U^{(j)})_t) \quad a.s.$$
		and as 	$(\cE(U^{(j)})_t)_{t\geq 0}$ uniquely solves \eqref{eq_deg10b} this implies $\eta_t^{(j)}=c_j U^{(j)}_t$ a.s. for all $t<T_1$ which prolonges to $t\geq 0$ due to the L\'evy properties. Thus 
		\begin{equation}\label{eq_deg09}
		\eta^{(j)}_t=- c_j U^{(j)}_t= - c_j \Log (e^{-\xi^{(j)}_t} )
		\quad \text{for all } j\in\cS,
		\end{equation} 
		is a necessary condition for \eqref{eq_deg00}.\\
		Further, if $\cS=\{1\}$, then the computation leading to \eqref{eq_deg09} extends to all $t\geq 0$ and there is nothing more to show. Thus assume $\cS$ consists of at least two different states such that $T_1<\infty$ $\PP_\pi$-a.s. due to the recurrency of $(J_t)_{t\geq 0}$. Then from \eqref{eq_deg00} $\PP_j$-a.s. for all $j\in\cS$  with $k=J(T_1)$
		\begin{align*}
		\int_{(0,T_1)}e^{-\xi^{(j)}_{s-}}d\eta^{(j)}_s + e^{-\xi^{(j)}_{T_1-}}Z^{(j,k)}_1(\eta)
		&= \int_{(0,T_1]}e^{-\xi_{s-}}d\eta_s  
		= c_{J_0}-c_{J_{T_1}}e^{-\xi_{T_1}} \\
		&= c_j - c_k e^{-\xi_{T_1-} - Z^{(j,k)}_1(\xi)}, 
		\end{align*}
		where $Z^{(i,j)}_1(\xi)=\Delta \xi_{T_1}$ and $Z^{(i,j)}_1(\eta)=\Delta \eta_{T_1}$ are the first and second component of $Z^{(i,j)}_1$ in \eqref{eq_MAP02}, respectively. But as we already know that  
		$$\int_{(0,T_1)}e^{-\xi^{(j)}_{s-}}d\eta^{(j)}_s = c_{j}-c_{j}e^{-\xi_{T_1-}}, \quad \PP_j\text{-a.s. for all }j\in\cS, $$
		we conclude $\PP_j$-a.s. for all $j\in\cS$ 
		\begin{align*}
		c_j - c_k e^{-\xi_{T_1-} - Z^{(j,k)}_1(\xi)} - e^{-\xi^{(j)}_{T_1-}}Z^{(j,k)}_1(\eta) &= c_{j}-c_{j}e^{-\xi_{T_1-}}\\
		\Rightarrow	c_k e^{- Z^{(j,k)}_1(\xi)} +  Z^{(j,k)}_1(\eta) &= c_{j}
		\end{align*}
		which yields that necessarily for all $(i,j)$ in $\cS^2$, $i\neq j$,
		\begin{equation}
		\label{eq_deg12}
		\supp(F^{(i,j)})=\{(x,y)\in\RR^2: y=c_i-c_j e^{-x} \}.
		\end{equation}
		Finally, to show \eqref{eq_deg13} note first that by \cite[Thm. II.37]{protter} it follows directly from the definition of $(U_t)_{t\geq 0}$ in \eqref{eq_deg14} that $\cE(U)_t=e^{-\xi_t}$. Further it is clear that 
		$$\Delta U_t + 1 = e^{-\Delta \xi_t} \quad \text{for all } t\geq 0,$$
		such that \eqref{eq_deg12} implies for all $n\in\NN$
		\begin{align*}
		\Delta  \eta_{T_n} &= c_{J_{T_n-}} - c_{J_{T_n}} e^{-\Delta \xi_{T_n}} = c_{J_{T_n-}}  - c_{J_{T_n}} (1+ \Delta U_{T_n})= - c_{J_{T_n}}\Delta U_{T_n} - \Delta c_{J_{T_n}}.
		\end{align*}
		Together with \eqref{eq_deg09} this yields 
		\begin{align*}
		\eta_t &= -\int_{(0,t]} c_{J_s}dU^{(J_s)}_s + \sum_{n: T_n\leq t} \Delta \eta_{T_n} = -\int_{(0,t]} c_{J_s}dU^{(J_s)}_s - \sum_{n: T_n\leq t}( c_{J_{T_n}}\Delta U_{T_n} + \Delta c_{J_{T_n}})\\
		&= -\int_{(0,t]} c_{J_s}dU_s  - \int_{(0,t]} dc_{J_s},
		\end{align*}
		which is \eqref{eq_deg13}.\\	
		The converse can be shown by direct computation. 
	\end{proof}

	\begin{remark}
		From Equation \eqref{eq_deg13} it follows that in the case $(X_t^{(2)},Y^{(2)}_t)\equiv 0$ (that is $\Delta \xi_{T_n}=0=\Delta \eta_{T_n}$ $\PP_\pi$-a.s. for all $n$) the sequence $\{c_i,i\in\cS\}$ has to be constant equal to some $c\in\RR$ and \eqref{eq_deg13} can be simplified to
		$$\eta_t= - c U_t= - c \,\Log (e^{-\xi_t})  \quad \text{for all } t\geq 0,$$
		which is equivalent to \eqref{eq_EM3} in the L\'evy case where $\cS=\{1\}$ (see e.g. \cite[Cor. 2.3]{behmelindnermaller}). \\
		If \eqref{eq_deg13} holds and $\Delta \xi_{T_n}=0$ $\PP_\pi$-a.s. for all $n$, then $F^{(i,j)}_\eta$ degenerates to an atom in $c_i-c_j$, while if $\Delta \eta_{T_n}=0$ $\PP_\pi$-a.s.  for all $n$, then $F^{(i,j)}_\xi$ degenerates to an atom in $\log \frac{c_j}{c_i}$ which implies that this can only happen if either $c_i>0, \forall i\in \cS$, or  $c_i<0, \forall i\in \cS$, (or $c_i\equiv 0$ which would correspond to the trivial and excluded case $\eta \equiv 0$). 
	\end{remark}

	\subsection{Proof of Theorem \ref{thmnessandsuff}: Almost sure convergence}\label{s42}
	
	For the proof of the convergence statements in Theorem \ref{thmnessandsuff} we need the following lemma. Note that the introduced process $\hat{\xi}$ is obtained from $\xi$ by  ``conflating'' the excursions of $\xi$ for $t\not\in\TT_j$ to single jumps and identifying the $n$-th exit and $n$-th return time of $j$. Other appearing processes will be conflated likewise when this is necessary.
	
	\begin{lemma}\label{lem-help1}
		Fix $j\in \cS$ and recall the sojourn time $\TT_j:=\{t\geq 0: J_t=t\}.$
		Then under $\PP_j$ the \emph{conflated process} $(\hat{\xi}_t)_{t\geq 0}:=(\hat{\xi}^j_t)_{t\geq 0}$ given by
		$$\hat{\xi}^j_t:= \xi_{t+\sum_{k=1}^n (\tau_k(j)-\tau^-_k(j))} \quad \text{for} \quad t\in \Big[\tau_n^-(j) - \sum_{k=1}^{n-1} (\tau_k(j)-\tau^-_k(j)), \tau_{n+1}^-(j) - \sum_{k=1}^n (\tau_k(j)-\tau^-_k(j))\Big), $$
		is a L\'evy process with characteristic triplet $(\gamma_{\xi^{(j)}}, \sigma^2_{\xi^{(j)}}, \hat{\nu})$,
		where $$\hat{\nu}(dx)=\nu_{\xi^{(j)}}(dx)-q_{j,j}\PP_j(\xi_{\tau_1(j)}-\xi_{\tau_1^-(j)}\in dx).$$
	\end{lemma} 
	\begin{proof}
	Clearly $\hat{\xi}_0=0$ $\PP_j$-a.s. since $\xi_0=0$ $\PP_j$-a.s. Further for $t\in \TT_j$ the process $\xi_t$ equals in law a L\'evy process with triplet $(\gamma_{\xi^{(j)}}, \sigma^2_{\xi^{(j)}}, \nu_{\xi^{(j)}})$ and thus $\hat{\xi}$ inherits independent increments and c\`adl\`ag paths of $\xi$. Stationarity of the increments follows from a standard property of MAPs, namely by \cite[Eq. XI.2.1]{asmussen}
	\begin{align*}
	\EE_j[f(\hat{\xi}_{t+s}-\hat{\xi}_t))]
	&= \EE_j[f(\xi_{t+s+ \sum_{k=1}^{n(s+t)} (\tau_k(j)-\tau^-_k(j))} -\xi_{t+\sum_{k=1}^{n(t)} (\tau_k(j)-\tau^-_k(j) )})]\\
	&= \EE_j\big[\EE_j[f(\xi_{t+s+ \sum_{k=1}^{n(s+t)} (\tau_k(j)-\tau^-_k(j))} -\xi_{t+\sum_{k=1}^{n(t)} (\tau_k(j)-\tau^-_k(j) )})| \cF_{t+\sum_{k=1}^{n(t)} (\tau_k(j)-\tau^-_k(j) )}]\big]\\
	&= \EE_{j}[f(\xi_{s+\sum_{k=n(t)+1}^{n(t+s)}(\tau_k(j)-\tau^-_k(j) ) })], \qquad \forall f\in C_c(\RR), 
	\end{align*}
	where $\sum_{k=n(t)+1}^{n(t+s)}(\tau_k(j)-\tau^-_k(j)) \overset{d}= \sum_{k=1}^{n(s)}(\tau_k(j)-\tau^-_k(j))$, such that we conclude
	\begin{align*}
	\EE_j[f(\hat{\xi}_{t+s}-\hat{\xi}_t))]&= \EE_j[f(\hat{\xi}_{s})] \qquad \forall f\in C_c(\RR).
	\end{align*}	
	Finally, the form of the jump measure $\hat{\nu}$ results from adding the jumps due to conflation which happen at rate $-q_{j,j}$ to the L\'evy measure $\nu_{\xi^{(j)}}$.
	\end{proof}	
	
	We also observe the following useful solidarity property.
	\begin{lemma}\label{lem-help2}
		Consider the process $(\hat{\xi}^j_t)_{t\geq 0}$ as in Lemma \ref{lem-help1}. Then
		$\lim_{t\to\infty} \hat{\xi}^j_t = \infty$ $\PP_j$-a.s. for some $j\in \cS$ if and only if $\lim_{t\in \TT_j, t\to \infty} \xi_t=\infty$ $\PP_j$-a.s. for all $j\in \cS$.
	\end{lemma}
	\begin{proof}
	Assume $\lim_{t\to\infty} \hat{\xi}^j_t= \infty$ $\PP_j$-a.s. for some $j\in\cS$, then clearly the subsequence $\xi_{\tau_{n}(j)}$ tends to $\infty$ $\PP_j$-a.s. as well for $n\to\infty$. By \cite[Lemma 7.1]{alsmeyerbuckmann2} it follows that $\lim_{n\to\infty} \xi_{\tau_n(j)}=\infty$ for all $j\in\cS$. 
	Now note that the sequence $(\xi_{\tau_n(j)})_{n\in\NN}$ takes exactly the same values as $(\hat{\xi}^j_{S_t})_{t\geq 0}$ where $(S_t)_{t\geq 0}$ is a subordinator with exponentially distributed jumps, namely $S_t=\sum_{i=1}^{M_t} (\tau_i^-(j)-\tau_{i-1}(j))$ for an arbitrary Poisson process $(M_t)_{t\geq 0}$. By \cite[Thm. 3.2(IV)]{ericksonmaller05b} this implies the claim. 
	\end{proof}
	
	We now prove the statement on almost sure convergence in Theorem \ref{thmnessandsuff}, that is we show:
	
	\begin{proposition}
	Assume $\lim_{t\in \TT_j, t\to \infty} \xi_t=\infty$ $\PP_j$-a.s. for some hence all $j\in \cS$. The exponential integral $\frace_{(\xi,\eta)}(t)$ converges $\PP_\pi$-a.s. as $t\to \infty$ to some random variable $\frace_{(\xi,\eta)}^\infty$ if and only if  \eqref{eq-nesssuffas} holds 
	for all $j\in \cS$.
	\end{proposition}
	
	\begin{proof}
		We start with the ``if'' statement. Fix any $j\in\cS$ and let $N_t=\sum_{n\in\NN} \mathds{1}_{\tau_n(j)\leq t}$ for $t\geq 0$. Then for any $t\geq 0$
		\begin{align*}
		\frace_{(\xi,\eta)}(t) &= \int_{(0,\tau_{N_t}(j)]} e^{-\xi_{s-}}d\eta_s + \int_{(\tau_{N_t}(j),t]} e^{-\xi_{s-}}d\eta_s \\
		&= \sum_{k=1}^{N_t} e^{\xi_{\tau_{k-1}(j)}}\int_{(\tau_{k-1}(j),\tau_k(j)]} e^{-(\xi_{s-} - \xi_{\tau_{k-1}(j)})}d\eta_s + e^{-\xi_{\tau_{N_t}(j)}} \int_{(\tau_{N_t}(j),t]} e^{-(\xi_{s-}-\xi_{\tau_{N_t}(j)})}d\eta_s\\
		&=:I_1+I_2.
		\end{align*}
		Here, $I_1$ is a classical perpetuity, namely $I_1=\sum_{k=1}^{N_t}( \prod_{\ell=1}^{k-1} A_\ell )B_k$ with
		$$(A_k,B_k)= (e^{-(\xi_{\tau_k(j)}-\xi_{\tau_{k-1}(j)})}, \int_{(\tau_{k-1}(j),\tau_k(j)]} e^{-(\xi_{s-}-\xi_{\tau_{k-1}(j)})} d\eta_s),$$
		such that $(A_k,B_k)_{k\in \NN}$ is an i.i.d. sequence under $\PP_j$.  
		Thus by \cite[Thm. 2.1]{goldiemaller00} $I_1$ converges a.s. since $\xi_{\tau_n(j)}\to \infty$ a.s. and 
		\begin{align*}\lefteqn{\int_{(1,\infty)} \frac{\log q}{\int_{(0,\log q]} \PP_j(\xi_{\tau_1(j)}>u) du } \PP_j\left( \Big|\int_{(0,\tau_1(j)]} e^{-\xi_{s-}} d\eta_s  \Big| \in dq \right)}\\
		&\leq \int_{(1,\infty)} \frac{\log q}{\int_{(0,\log q]} \PP_j(\xi_{\tau_1(j)}>u) du } \PP_j\left(\sup_{0<t\leq \tau_1(j)} \Big|\int_{(0,\tau_1(j)]} e^{-\xi_{s-}} d\eta_s  \Big| \in dq \right) <\infty \end{align*}
		by our assumptions.\\
		To find an appropriate bound for $I_2$ note first that
		$$|I_2| \leq e^{-\xi_{\tau_{N_t}(j)}} \sup_{\tau_{N_t}(j) < t \leq \tau_{N_t+1}(j)} \Big|\int_{(\tau_{N_t(j)},t]} e^{-(\xi_{s-}-\xi_{\tau_{N_t(j)}})} d\eta_s  \Big|=: e^{-\xi_{\tau_{N_t}(j)}} W_{N_t},$$
		where $(W_n)_{n\in\NN}$ is an i.i.d. sequence. 
		We will show that for some $c>0$
		$$\lim_{t\to\infty} e^{c\tau_{N_t(j)}} e^{-\xi_{\tau_{N_t}(j)}} W_{N_t}=0 \quad \PP_j\text{-a.s.},$$
		which is equivalent to state that
		\begin{equation}\label{eq-almostsure1}  \lim_{t\to\infty} (\xi_{\tau_{N_t}(j)} - c\tau_{N_t(j)} - \log^+ W_{N_t})=\lim_{t\to\infty} \xi_{\tau_{N_t}(j)}\big( 1- c \frac{\tau_{N_t(j)}}{\xi_{\tau_{N_t}(j)}} - \frac{\log^+ W_{N_t}}{\xi_{\tau_{N_t}(j)}} \big)= \infty\quad \PP_j\text{-a.s.}\end{equation} 
		Since $\lim_{t\in\TT_j,t\to\infty}\xi_t=\infty$ and since by Lemma \ref{lem-help1} the conflated process $(\hat{\xi}_t^j)_{t\geq 0}$ is a L\'evy process under $\PP_j$, it holds that $0<\EE_j[\hat{\xi}_1^j]\leq \infty$. This clearly implies 
		$0<\EE_j[\xi_{\tau_1(j)}]\leq \infty$ and we fix $c=\frac12 \EE_j[\xi_{\tau_1(j)}])$ whenever the appearing expectation is finite, and set $c=1$ otherwise.
		This then yields (in case of infinite expectation using Kesten's trichotomy \cite{Kesten}; also see \cite[Proof of Thm. 3.1]{alsmeyerbuckmann2})
		$$\limsup_{t\to\infty} \frac{\tau_{N_t(j)}}{\xi_{\tau_{N_t(j)}}} = \limsup_{t\to\infty}\frac{N_t(j)}{\xi_{\tau_{N_t(j)}}} \frac{\tau_{N_t(j)}}{N_t(j)} \leq \frac{1}{2c}\EE[\tau_1(j)]<\infty \quad \PP_j\text{-a.s.}$$
		Further, whenever $\EE_j[\log^+ W_1]+ \EE_j[\xi_{\tau_1(j)}]=\infty$, by \cite[Lemma 8.1]{alsmeyerbuckmann2} Equation \eqref{eq-nesssuffas} implies directly that
		\begin{equation}\label{eq-almostsure2}
		\limsup_{t\to \infty} \frac{\log^+ W_{N_t}}{\xi_{\tau_{N_t}(j)}} =0 \quad \PP_j\text{-a.s.}
		\end{equation} 
		On the other hand, if $\EE_j[\log^+ W_1]+ \EE_j[\xi_{\tau_1(j)}]<\infty$, then 
		$$0= \limsup_{t\to\infty} \frac{W_{N_t}}{t} \geq \limsup_{t\to\infty} \frac{\log^+ W_{N_t}}{t} =\limsup_{t\to\infty} \frac{\log^+ W_{N_t}}{\xi_{\tau_{N_t}(j)}} \frac{\xi_{\tau_{N_t}(j)}}{t} \quad \PP_j\text{-a.s.}$$
		which again implies \eqref{eq-almostsure2} since $\lim_{t\to\infty} \xi_{\tau_{N_t(j)}}/t >0$. 
		Hence \eqref{eq-almostsure1} follows and the growth of $|I_2|$ is bounded by $e^{-ct}$ which proves the almost sure convergence under $\PP_j$. As $j\in\cS$ was arbitrary this implies $\PP_\pi$-a.s. convergence.
					
		For the ``and only if'' statement, assume \eqref{eq-nesssuffas} fails. Then $\EE_j[\log^+ W_1]=\infty$ as well, since $\EE_j[\xi_{\tau_1(j)}^+ \wedge \log q]$ is bounded away from zero by assumption. Thus by \cite[Lemma 8.1]{alsmeyerbuckmann2} it follows that
		\begin{equation}\label{eq-almostsure3}
		\limsup_{t\to \infty} \frac{\log^+ W_{N_t}}{\xi_{\tau_{N_t}(j)}} =\infty \quad \PP_j\text{-a.s.}
		\end{equation}
		such that
		\begin{align*}
		\limsup_{t\to\infty} \log^+(e^{-\xi_{\tau_{N_t}(j)}} W_{N_t}) 
			&= \limsup_{t\to\infty} (-\xi_{\tau_{N_t}(j)} + \log^+ W_{N_t}) \\
			&=  \limsup_{t\to\infty} \xi_{\tau_{N_t}(j)}\big(-1 + \frac{\log^+ W_{N_t}}{\xi_{\tau_{N_t}(j)}}\big)\\
			&=\infty,
		\end{align*}
		and $|I_2|$ is not converging. Thus, by  conditional independence  given $(J_t)_{t\geq 0}$ of $I_1$ and $I_2$, also $\frace_{(\xi,\eta)}(t)$ does not converge $\PP_j$-a.s. as we had to show.
	\end{proof}

	\subsection{Proof of Theorem \ref{thmnessandsuff}: Convergence in probability}\label{s43}
	
	In this section we prove:	
	
	\begin{proposition}\label{prop-weak}
		If $\lim_{t\in \TT_j, t\to \infty} \xi_t=\infty$ $\PP_j$-a.s. and \eqref{eq-nesssuffweak}
		  hold for some $j\in \cS$, then for all $j\in \cS$ 
		  $\frace_{(\xi,\eta)}(t)\to \frace_{(\xi,\eta)}^\infty$ in $\PP_j$-probability.
	\end{proposition}

	\begin{proof}
		Fix $j\in \cS$ such that $\lim_{t\in \TT_j, t\to \infty} \xi_t=\infty$ and \eqref{eq-nesssuffweak} hold. Under $\PP_j$ we split up the exponential integral as follows:
		\begin{align}
		\frace_{(\xi,\eta)}(t)&= \int_{(0,t]} e^{-\xi_{s-}} \mathds{1}_{s\in \TT_j} d\eta_s  +\int_{(0,t]} e^{-\xi_{s-}} \mathds{1}_{s\not\in \TT_j} d\eta_s  \nonumber \\
		&= \sum_{k=1}^{N_t^-} \int_{(\tau_{k-1}(j),\tau_k^-(j))} e^{-\xi_{s-}} d\eta_s +  \sum_{k=1}^{N_t} \int_{[\tau_k^-(j),\tau_k(j)]} e^{-\xi_{s-}} d\eta_s + \begin{cases}
		\int_{(\tau_{N_t}(j),t]}e^{-\xi_{s-}} d\eta_s, & t\in \TT_j,\\
		\int_{(\tau^-_{N_t+1}(j),t]}e^{-\xi_{s-}} d\eta_s, & t\not\in \TT_j,
		\end{cases} \label{eq-weakperturb}
		\end{align}
		where $N_t:=N_t(j):=\sum_{n\in\NN} \mathds{1}_{\tau_n(j)\leq t}$ and $N_t^-:=N^-_t(j):=\sum_{n\in\NN} \mathds{1}_{\tau^-_n(j)\leq t}$ count the returns to and exits from $j$ up to time $t$, respectively.\\
		Define 
		$$\tilde{F}_k:=\int_{[\tau_k^-(j),\tau_k(j)]} e^{-(\xi_{s-}-\xi_{\tau_k^-(j)})} d\eta_s = \int_{(\tau_k^-(j),\tau_k(j)]} e^{-(\xi_{s-}-\xi_{\tau_k^-(j)})} d\eta_s + \Delta \eta_{\tau_k^-(j)},\quad k\in\NN,$$
		and 
		$$F_t:=\sum_{k=1}^{N_t^-} \tilde{F}_k, \quad t\geq 0,$$
		then clearly $(\tilde{F}_k)_{k\in\NN}$ forms an i.i.d. sequence and the conflated  process $(\hat{F}_t)_{t\geq 0}$ (in the same sense as in Lemma \ref{lem-help1}) is a compound Poisson process since the sojourn times $\tau_k^-(j)-\tau_{k-1}(j)$, i.e. the interarrival times of $(\hat{F}_t)_{t\geq 0}$, are exponentially distributed. Further
		$$\sum_{k=1}^{N_t} \int_{[\tau_k^-(j),\tau_k(j)]} e^{-\xi_{s-}} d\eta_s = \sum_{k=1}^{N_t} e^{-\xi_{\tau_k^-(j)}} \tilde{F}_k =\int_{(0,\tau_{N_t}(j)]} e^{-\xi_{s-}}dF_s = \int_{(0,t]} e^{-\xi_{s-}}dF_s.$$
		Thus for any $t\in \TT_j$ 
		\begin{align*}
		\frace_{(\xi,\eta)}(t) 
		&= \int_{(0,t]} e^{-\xi_{s-}} \mathds{1}_{s\in \bigcup_{k\in\NN} (\tau_{k-1}(j),\tau_k^-(j))} d\eta_s  +  \int_{(0,t]} e^{-\xi_{s-}}d\hat{F}_s,
		\end{align*}
		and the conflated version $(\hat{\frace}^j_{(\xi,\eta)}(t) )$ of this process (which is constant on $\TT_j^c$ anyway) is an exponential integral of the bivariate L\'evy process $(\hat{\xi}_t^{j}, \hat{\hat{\eta}}_t^{j}+ \hat{F}_t)_{t\geq 0}$, where $\hat{\hat{\eta}}$ is a variant of $\hat{\eta}$ that has no jumps at conflation times. Since $\hat{\xi}_t^j \to \infty$ a.s. by assumption, we see from \cite[Thm. 2]{ericksonmaller05} (also see Equation \eqref{eq_EM1}) that $(\hat{\frace}^j_{(\xi,\eta)}(t) )$ converges almost surely under $\PP_j$ if and only if 
		\begin{equation} \label{eq-weakEMcondition}
		\int_{(1,\infty)} \left(\frac{\log y}{A_{\hat{\xi}^j}(\log y)} \right) |\bar{\nu}_{\hat{\hat{\eta}}^j+\hat{F}}(dy)|<\infty\end{equation}
		which is equivalent to \eqref{eq-nesssuffweak} by Lemma \ref{lem-help1}. \\
		It remains to show that for $t\notin \TT_j$ the appearing perturbation term $\int_{(\tau^-_{N_t+1}(j),t]}e^{-\xi_{s-}} d\eta_s$ in \eqref{eq-weakperturb} is bounded appropriately. To see this, set
		$$M_t:=\tilde{M}_{N_t}:=\sup_{s\in (\tau^-_{N_t}(j),\tau_{N_t}(j)]} \big|\int_{(\tau^-_{N_t}(j), s]} e^{-(\xi_{u-} - \xi_{\tau^-_{N_t}(j)})} \big| d\eta_u $$
		such that 
		\begin{align*}
		\PP_j(M_t\leq x)&= \PP_j(\tilde{M}_{N_t}\leq x) = \sum_{k\in\NN_0} \PP_j(\tilde{M}_{k}\leq x, N_t=k)\\
		&= \sum_{k\in\NN_0} \PP_j(\tilde{M}_{k}\leq x, \tau_k(j)\leq t, \tau_{k+1}(j)>t)\\
		&= \sum_{k\in\NN_0} \int_{(0,t]} \PP_j(\tilde{M}_{k}\leq x, \tau_k(j)-\tau_k^-(j) \leq t-s, \tau_{k+1}(j)-\tau_k^-(j) >t-s) \PP_j(\tau_k^-(j)\in ds)\\
		&=  \int_{(0,t]} \PP_j(\tilde{M}_{1}\leq x, \tau_1(j)-\tau_1^-(j) \leq t-s, \tau_{2}(j)-\tau_1^-(j) >t-s) \Big(\sum_{k\in\NN_0} \PP_j(\tau_k^-(j)\in ds) \Big).
			\end{align*}
		This convolution integral has a distributional limit by the key renewal theorem \cite[Thm. V.4.3]{asmussen}, i.e.
		\begin{align*}
		\PP_j(M_t\leq x)&\overset{d}{\underset{t\to \infty}\longrightarrow} \frac{1}{\EE_j[\tau_1^-(j)]} \int_0^\infty \PP_j(\tilde{M}_{1}\leq x, \tau_1(j)-\tau_1^-(j) \leq s, \tau_{2}(j)-\tau_1^-(j) > s) ds.
 		\end{align*}
 		Hence
		\begin{align*}
		\frace_{(\xi,\eta)}(t) &= \mathds{1}_{t\in \TT_j} \int_{(0,t]} e^{-\xi_{s-}} d\eta_s\\
		&\quad +		
		\mathds{1}_{t\not\in \TT_j} \Big(\underbrace{\int_{(0,\tau^-_{N_t+1}]} e^{-\xi_{s-}} d\eta_s}_{\text{ converges } \PP_j\text{-a.s. as } \tau^-_{N_t+1}\in \bar{\TT}_j} +  \underbrace{e^{- \xi_{\tau^-_{N_t+1}(j)}} }_{\to 0 \,\PP_j\text{-a.s.}}\underbrace{\int_{(\tau^-_{N_t+1}(j), t]} e^{-(\xi_{s-} - \xi_{\tau^-_{N_t+1}(j)})}d\eta_s}_{\sup |\cdot | \,\text{converges in distribution}}\Big)
		\end{align*}
		converges in $\PP_j$-probability as $t\to \infty$ by Slutzky's theorem as claimed.\\
		Finally note that convergence under $\PP_{j'}$ follows due to the positive recurrency of $(J_t)$: After reaching state $j$ the exponential integral converges in probability as shown, while up to $\tau_1(j)$ it cannot diverge as this would imply divergence under $\PP_j$ by the same argument.
	\end{proof}
	
	\subsection{Proof of Theorem \ref{thmnessandsuff}: Divergence}\label{s44}
	
	We will prove divergence of the exponential integral $\frace_{(\xi,\eta)}$ in the two possible cases separately and start with:
	
	\begin{proposition}
		Assume that the degeneracy condition \eqref{eq_deg00} fails and $\liminf_{t\in \TT_j, t\to \infty}\xi_t<\infty$ for some $j\in \cS$, then 
		$$|\frace_{(\xi,\eta)}(t)|\overset{\PP_\pi}\longrightarrow \infty.$$	
	\end{proposition}
	\begin{proof}
		As seen in \eqref{eq_deg07} in the proof of Proposition \ref{prop_degchar}, whenever \eqref{eq_deg00} fails, we find $u>0$ such that for any sequence $\{c_i,i\in\cS\}$ 
		\begin{equation}\label{eq-diverge01}\PP_\pi\left(\int_{(0,u]} e^{-\xi_{s-}} d\eta_s = c_{J_0}-c_{J_u} e^{-\xi_u} \right)<1.\end{equation}
		Now consider
		\begin{align*}
		Z_n^u&:= \frace_{(\xi,\eta)}(nu)
		= \int_{(0,nu]} e^{-\xi_{s-}} d\eta_s
		=:\sum_{k=1}^n \left(\prod_{\ell=1}^{k-1} A^u_\ell\right) B^u_k,
		\end{align*}
		with
		$$(A^u_k,B^u_k)=( e^{-(\xi_{ku}-\xi_{(k-1)u})}, \int_{((k-1)u, ku]} e^{-(\xi_{s-}-\xi_{(k-1)u})}d\eta_s).$$
		Here $(A^u_k,B^u_k)_{k\in\NN}$ is a sequence of random vectors that is modulated by an ergodic Markov chain $(J^u_n)_{n\in \NN_0}$ which is a skeleton chain of $(J_t)_{t\geq 0}$.  Denoting the $k$-th return time of $(J^u_n)_{n\in\NN}$ to $j$ as
		$\tau_k^u(j)$, we note that $\prod_{\ell=1}^{\tau_k^u(j)} A_\ell^u = \exp(-\xi_{\tau_k^u(j) u})$ does not tend to $0$ a.s. for $k\to\infty$ due to our assumption. Together with \eqref{eq-diverge01} we thus conclude from \cite[Thm. 3.4]{alsmeyerbuckmann1} that $|Z_n^u|\overset{\PP_{\pi^u}}\longrightarrow \infty$, $n\to \infty$, where the invariant distribution $\pi^u$ of $(J^u_n)_{n\in \NN_0}$ is equivalent to $\pi$. \\
		Further, with $n_t:=\sup\{n\in \NN: nu\leq t\}$ and $r_t=t-n_t u \in[0,u)$, under $\PP_\pi$
		\begin{align*}
		\frace_{(\xi,\eta)}(t) &= \int_{(0,r_t]} e^{-\xi_{s-}}d\eta_s  + e^{-\xi_{r_t}}\int_{(r_t,t]} e^{-(\xi_{s-} - \xi_{r_t})}d\eta_s
		\overset{d}= \int_{(0,r_t]} e^{-\xi_{s-}}d\eta_s + e^{-\xi_{r_t}} (Z_{n_t}^u)',
		\end{align*}
	where $(Z_{n_t}^u)'$ is a copy of $Z_{n_t}^u$ that is independent of the past up to time $r_t$. Since $|Z_{n_t}^u|\overset{\PP_{\pi}}\longrightarrow \infty$, $t\to\infty$, while $e^{-\xi_{r_t}}$ is bounded away from zero and $\int_{(0,r_t]} e^{-\xi_{s-}}d\eta_s$ is finite, we observe that $|\frace_{(\xi,\eta)}(t)|\overset{\PP_\pi}\longrightarrow \infty$ as stated.	
	\end{proof}
	
	To complete the proof of Theorem \ref{thmnessandsuff} it remains to show:
	
	\begin{proposition}
		Assume that both the degeneracy condition \eqref{eq_deg00} and \eqref{eq-nesssuffweak} fail for all $j$, then 
		$$|\frace_{(\xi,\eta)}(t)|\overset{\PP_\pi}\longrightarrow \infty.$$
	\end{proposition}
\begin{proof}
	Assume \eqref{eq-nesssuffweak} fails for all $j\in\cS$. 
	Fixing $j$ we follow the lines of the proof of Proposition \ref{prop-weak} up to failure of \eqref{eq-weakEMcondition} and conclude by \cite[Thm. 2]{ericksonmaller05} that \begin{equation}\label{eq-divEM}
	\big|\hat{\frace}^j_{(\xi,\eta)}(t)\big| \overset{\PP_j}\longrightarrow \infty, \quad t\to \infty,\end{equation}
	whenever the conflated integral is not degenerate, i.e. if there is no constant $c_j\in\RR$ such that
	\begin{equation}\label{eq-degconflated}
	\hat{\frace}^j_{(\xi,\eta)} (t) = c_j - c_j e^{-\hat{\xi}_t} \quad \text{for all } t\geq 0 \quad\PP_j\text{-a.s.}\end{equation}
	This follows from failure of \eqref{eq_deg00} as \eqref{eq-degconflated} is either true for all $j\in\cS$ or none. More precisely, \eqref{eq-degconflated} is equivalent to $\hat{\eta}^j_t=-c_j \hat{U}_t^j$, which is a consequence of Proposition \ref{prop_degcond}. This in turn is equivalent to $\eta^{(j)}_t=-c_j U_t^{(j)}$ $\PP_j$-a.s. and 
	\begin{equation}\label{eq-degconflated2} \eta_{\tau_k(j)}-\eta_{\tau^-_k(j)}=-c_j (U_{\tau_k(j)}-U_{\tau^-_k(j)}) \quad \PP_j\text{-a.s.}\end{equation}
	However, if \eqref{eq-degconflated} fails for some $j'\neq j$, then \eqref{eq-degconflated2} necessarily fails as well.\\
	By the same argumentation as at the end of the proof of Proposition \ref{prop-weak}, the divergence \eqref{eq-divEM} implies divergence of $|\frace_{(\xi,\eta)}(t)|$ in $\PP_j$-probability. As $j$ was chosen arbitrarily this yields the result. 
\end{proof}

	\section{Sufficient conditions}\label{S4}
	\setcounter{equation}{0}
	
	Although Theorem \ref{thmnessandsuff} provides necessary and sufficient conditions for convergence of $\frace_{(\xi,\eta)}(t)$ it is hardly applicable as the given assumptions are difficult to check. Thus this section aims at additional, easy to check conditions for convergence of $\frace_{(\xi,\eta)}(t)$. In particular we will formulate  conditions in terms of the long term mean $\kappa_\xi$.
	
	To this end we decompose the exponential integral $\frace_{(\xi,\eta)}(t)$ using the L\'evy-It\^o-type decomposition \eqref{eq_MAP04} as follows.
	\begin{align}
	\frace_{(\xi,\eta)}(t)&= \int_{(0,t]} e^{-\xi_{s-}} d(\gamma^\eta_s+W^\eta_s+Y^{b,\eta}_s+Y^{s,\eta}_s + Y^{(2)}_s) \nonumber\\
	&= \int_{(0,t]} e^{-\xi_{s-}} d(\gamma^\eta_s + W^\eta_s+Y^{s,\eta}_s) + \int_{(0,t]} e^{-\xi_{s-}} d(Y^{b,\eta}_s+Y^{(2)}_s) \nonumber \\
	&=: \frace^{(1)}(t)+\frace^{(2)}(t).
	\label{eq_fracesplit}
	\end{align}
	
	We now treat the two exponential integrals in \eqref{eq_fracesplit} separately. First, to study $\frace^{(1)}$ we need the following technical lemma. 
	
	\begin{lemma}\label{lem_martingale}
		The process $(W_t^\eta+Y^{s,\eta}_t)_{t\geq 0}$ is a martingale. Furthermore $(W_t^\eta+Y^{s,\eta}_t)_{t\geq 0}$ is square-integrable if $\sup_{j\in\cS} (\sigma^2_{\eta^{(j)}} + \int_{(0,1)} x^2 \nu_{\eta^{(j)}}(dx) )  <\infty$.
	\end{lemma}
	\begin{proof}
		For the martingale property note that for $0\leq s\leq t$ 
		\begin{align*}
		\EE[W_t^\eta+Y^{s,\eta}_t|\cF_s]&= 	\EE\big[ \EE[W_t^\eta+Y^{s,\eta}_t| \cF_s, J_u=j(u), s<u\leq t] |\cF_s\big] \\
		&=  \EE\big[ \EE[W_s^\eta + Y^{s,\eta}_s | \cF_s, J_u=j(u), s<u\leq t] |\cF_s\big]\\
		&\quad 	+ \EE\big[ \EE[ \int_{(s,t]} \sigma^2_{\eta^{(j(u))}} dW_u  +  \lim_{\varepsilon \to 0} \int_{(s,t]} \int_{\varepsilon\leq |x| <1} x (N_{\eta^{(j(u))}}(du,dx) - du \,\nu_{\eta^{(j(u))}}(dx))  \\
		& \qquad \qquad | \cF_s, J_u=j(u), s<u\leq t] |\cF_s\big]\\
		&= W^\eta_s+Y^{s,\eta}_s,
		\end{align*}
		where in the last step we have used that $W_t^{\eta^{(j)}}$ and $Y_t^{s,\eta^{(j)}}$ are (square-integrable) martingales for any $j\in\cS$. Square-integrability of  $(W_t^\eta+Y^{s,\eta}_t)_{t\geq 0}$ under the given condition follows from \cite[Cor. II.3]{protter} and 
		\begin{align*}
		\EE[ \langle W^\eta+Y^{s,\eta}\rangle_t]&= \EE\left[\int_{(0,t]} d \langle W^{\eta^{(J_s)}}+Y^{s,\eta^{(J_s)}} \rangle_s \right]\\
		&= \EE\left[\int_{(0,t]} (\sigma^2_{\eta^{(J_s)}} + \int_{(0,1)} x^2 \nu_{\eta^{(J_s)}}(dx) )  ds \right]\\
		&\leq t\cdot \sup_{j\in\cS} \left(\sigma^2_{\eta^{(j)}} + \int_{(0,1)} x^2 \nu_{\eta^{(j)}}(dx) \right) <\infty.
		\end{align*}
	\end{proof}
	
	Following ideas from \cite{ericksonmaller05} we now show a.s. convergence of $\frace^{(1)}(t)$ as $t\to \infty$ under rather weak conditions.
	
	\begin{proposition}\label{prop_frace1strongconv}
		Assume that $0<\kappa_\xi<\infty$ and 
\begin{equation}\label{eq_proofs0} \sup_{j\in\cS} \left(|\gamma_{\eta^{(j)}}|+ \sigma^2_{\eta^{(j)}} + \int_{(0,1)} x^2 \nu_{\eta^{(j)}} (dx) \right) <\infty.\end{equation} Then 
		$\frace^{(1)}(t)$ converges $\PP_\pi$-a.s. to a finite random variable as $t\to \infty$.
	\end{proposition}
\begin{proof}
	It is sufficient to prove convergence of the given integral over some interval $(L,\infty)$, with possibly random $L\in[0,\infty)$. To find a suitable $L$ fix some constant $c\in(0,\kappa)$ and set $L:= \sup\{t\geq 0:\, \xi_{t-} - ct \leq 0 \}$ if the set is not empty and $L:=0$ otherwise. Then $L$ is a random variable such that $\xi_t\geq ct$  for all $t>L$ and it remains to consider
		\begin{align}\label{eq_proofs1}
	  \int_{(L,\infty)} e^{-\xi_{s-}} d(\gamma^\eta_s + W^\eta_s+Y^{s,\eta}_s) &= \int_{(L,\infty)} e^{-\xi_{s-}} d\gamma^\eta_s + \int_{(L,\infty)} e^{-\xi_{s-}}d(W^\eta_s+Y^{s,\eta}_s),
	\end{align}
	where 
	\begin{align*}
	\left|\int_{(L,\infty)} e^{-\xi_{s-}} d\gamma^\eta_s\right| &= \left|\int_{(L,\infty)} e^{-\xi_{s-}} \gamma_{\eta^{(J_s)}} ds \right| \leq \int_{(L,\infty)} e^{-\xi_{s-}} |\gamma_{\eta^{(J_s)}}| ds \\
	&\leq \sup_{j\in\cS} \left(|\gamma_{\eta^{(j)}}|\right)  \int_{(L,\infty)} e^{-\xi_{s-}} ds
	\leq \sup_{j\in\cS} \left(|\gamma_{\eta^{(j)}}|\right)  \int_{(L,\infty)} e^{-cs} ds\\
	&<\infty.
	\end{align*}
	For the second integral in \eqref{eq_proofs1} define $\lambda_t:=\xi_t \vee ct$, then $\lambda_{t-} \geq ct$ for all $t\geq 0$ and $\lambda_{t-}=\xi_{t-}$ for all $t>L$, such that in particular 
	$$\lim_{t\to \infty} \int_{(L,L\vee t]} e^{-\xi_{s-}}d(W^\eta_s+Y^{s,\eta}_s) = \lim_{t\to \infty} \int_{(L,L\vee t]} e^{-\lambda_{s-}}d(W^\eta_s+Y^{s,\eta}_s) \qquad \PP_\pi\text{-a.s.}$$
	By Lemma \ref{lem_martingale} the process $W^\eta_s+Y^{s,\eta}_s$ is a square-integrable martingale with mean $0$ and quadratic variation
	\begin{align*}
	\langle W^\eta+Y^{s,\eta}\rangle_t &= \int_{(0,t]} (\sigma^2_{\eta^{(J_s)}} + \int_{(0,1)} x^2 \nu_{\eta^{(J_s)}} (dx) )  ds =: \int_{(0,t]} \rho(J_s) ds.
	\end{align*}
	Thus using It\^o's isometry 
	\begin{align*}
	\EE_\pi\left[\left(\int_{(0,t]}  e^{-\lambda_{s-}}d(W^\eta_s+Y^{s,\eta}_s) \right)^2 \right]
	&= \int_{(0,t]}  \EE\left[e^{-2\lambda_{s-}}  \right] d\langle W^\eta+Y^{s,\eta}\rangle_s
	= \int_{(0,t]}  \EE\left[e^{-2\lambda_{s-}}  \right]\rho(J_s) ds\\
	&\leq \int_{(0,t]} e^{-2cs} \rho(J_s) ds \leq \sup_{j\in\cS} (\rho(J_s))   \int_{(0,t]} e^{-2cs} ds\\
	&\leq \frac{1}{2c} \sup_{j\in\cS} (\rho(J_s))   <\infty,
	\end{align*}
	such that $t\mapsto \int_{(0,t]} e^{-\lambda_{s-}}d(W^\eta_s+Y^{s,\eta}_s)$ is a  martingale with bounded, converging second moments. It therefore converges $\PP_\pi$-a.s. as $t\to\infty$ which yields the claim.
\end{proof}

\begin{remark}
	The above obtained sufficient condition for convergence of $\frace^{(1)}$ is not optimal and only chosen for presentation here as it is easy to check and interpret. If needed, necessary and sufficient conditions for convergence of $\frace^{(1)}$ could as well be obtained by applying Theorem \ref{thmnessandsuff} in this case. 
\end{remark}

	Clearly, for $\cS$ finite, in Proposition \ref{prop_frace1strongconv} we can drop the assumptions \eqref{eq_proofs0} and $\kappa_\xi<\infty$. Nevertheless, for countable $\cS$, \eqref{eq_proofs0} is not redundant as will be outlined by the following example.

\begin{example}
	Consider the petal flower Markov process $(J_t)_{t\geq 0}$ as defined in Example \ref{ex-xinottoinfty}.
	Choose $\xi$ and $\eta$ to be conditionally independent with $Y_t^{(2)}\equiv0$ and
	$$\xi_t=X_t^{(2)} = \sum_{n\geq 1} \sum_{i,j\in\NN} Z_n^{(i,j)} \mathds{1}_{\{J_{T_n-}=i,J_{T_n}=j, T_n\leq t\}}$$
	where 
	$$Z_n^{(i,j)}:=Z^{(i,j)}:=\begin{cases}
	2, & j=1,\\
	0, & \text{otherwise.}
	\end{cases}$$
	Then $\xi_t\to \infty$ $\PP_\pi$-a.s. as $t\to\infty$ with 	\begin{align*}
	\kappa_\xi & = \sum_{\substack{(i,j)\in \NN\times \NN \\ i\neq j}} \pi_i  q_{i,j} \EE[Z^{(i,j)}] 
	= 2 \sum_{i\in\NN\backslash\{1\}} \pi_i q_{i,1} = 2 \sum_{i\in\NN\backslash\{1\}} \frac{q_{1,i}}{2q} q
	= q.
	\end{align*}
	Further, setting $\eta_t=\int_{(0,t]} \gamma_{\eta^{(J_s)}} ds$ with 
	$$\gamma_{\eta^{(j)}}=\begin{cases}
	0,& j=1,\\
	\exp\left(p_{1,j}^{-1} \right),& \text{otherwise},
	\end{cases}$$ 
	we compute under $\PP_1$
	\begin{align*}
	\int_{(0,t]} e^{-\xi_{s-}} d\eta_s &= \int_{(0,t]} e^{-\xi_{s-}} \gamma_{\eta^{(J_s)}} ds
	= \int_{(0,t]} \exp\left(-\xi_{s-} + p_{1,J_s}^{-1} \mathds{1}_{\{J_s\neq 1\}}\right)  ds
	\end{align*}
	which diverges, as by an argumentation as in Example \ref{ex-xinottoinfty} using the Borel-Cantelli lemma
	$$\limsup_{t\to\infty} (-\xi_{t-} + p_{1,J_t}^{-1} \mathds{1}_{\{J_t\neq 1\}}) =\infty \quad \PP_1\text{-a.s.}$$			
\end{example}

The next proposition gives conditions for almost sure convergence and convergence in probability of the exponential integral $\frace^{(2)}$ as defined in \eqref{eq_fracesplit}.

\begin{proposition}\label{prop_frace2conv}
	Assume $0<\kappa_\xi<\infty$.
	\begin{enumerate}
		\item The exponential integral $\frace^{(2)}(t)$ converges $\PP_\pi$-a.s. to a finite random variable as $t\to \infty$ if and only if 	
		\begin{equation}
		\label{eq_proofs4}
		\int_{(1,\infty)} \log q \; \PP_j\Big(\sup_{0< t\leq \tau_1(i)} e^{-\xi_{t-}} | \Delta(Y^{b,\eta}_t+ Y^{(2)}_t) |\in dq\Big)<\infty \quad \text{for all} \quad j\in \cS.
		\end{equation} 
		\item The exponential integral $\frace^{(2)}(t)$ converges in $\PP_j$-probability to some random variable $\frace^{(2)}_\infty$ as $t\to\infty$, 
		if and only if \begin{equation}
		\label{eq_proofs5}
		\int_{(1,\infty)} \log q \; \PP_j\left(\left|\int_{(0,\tau_1(i)]} e^{-\xi_t-} d(Y^{b,\eta}_t+ Y^{(2)}_t) \right|\in dq\right)<\infty.
		\end{equation} 
		Moreover, in this case convergence in $\PP_j$-probability holds for all $j\in \cS$.
	\end{enumerate}
\end{proposition}
\begin{proof}
	Assume $Y^{b,\eta}_t+Y^{(2)}_t\not\equiv 0$ as otherwise $\frace^{(2)}(t)\equiv 0$ a.s. and there is nothing to show.\\
Let $\{\tilde{T}_n, n\in\NN_0\}$ be the jump times of $(Y^{b,\eta}_t+J_t)_{t\geq 0}$ with $\tilde{T}_0:=0$ and set $\tilde{N}_t=\sum_{n=1}^\infty \mathds{1}_{\tilde{T}_n\leq t}$. Then  $\{T_n,n\in\NN_0\} \subseteq \{\tilde{T}_n, n\in\NN_0\}$ and further $\{\tilde{T}_n, n\in\NN_0\}$ contains all jump times of $(Y^{b,\eta}_t+Y^{(2)}_t)_{t\geq 0}$. Thus we can reformulate
	\begin{align}
	\frace^{(2)}(t)&= \int_{(0,t]} e^{-\xi_{s-}} d(Y^{b,\eta}_s+Y^{(2)}_s) = \sum_{i=1}^{\tilde{N}_t} e^{-\xi_{\tilde{T}_i-}} \Delta(Y^{b,\eta}_{\tilde{T}_i}+ Y^{(2)}_{\tilde{T}_i}) \nonumber \\
	&= \sum_{i=1}^{\tilde{N}_t} \prod_{k=1}^{i} e^{-(\xi_{\tilde{T}_{k}-}-\xi_{\tilde{T}_{k-1}-})} \Delta(Y^{b,\eta}_{\tilde{T}_i}+ Y^{(2)}_{\tilde{T}_i}) \nonumber \\
	&=: \sum_{i=1}^{\tilde{N}_t} \left(\prod_{k=1}^{i} A_k \right) B_i, \label{eq_proofsperp}
	\end{align}
	where 
	$$(A_n,B_n)_{n\in\NN}=\left(e^{-(\xi_{\tilde{T}_{n}-}-\xi_{\tilde{T}_{n-1}-})}, \Delta(Y^{b,\eta}_{\tilde{T}_n}+ Y^{(2)}_{\tilde{T}_n}) \right)_{n\in\NN}$$
	is a sequence of random vectors modulated by a Markov chain $(\tilde{J}_n)_{n\in\NN}$ with state space $\cS$. Hereby $\tilde{J}$ is a retarded discrete time version of $J$ whose return times 
	$$ \tilde{\tau}_0(j):=0,\quad \text{and} \quad 	\tilde{\tau}_n(j):=\inf\{k>\tilde{\tau}_{n-1}(j): \tilde{J}_k=j, \tilde{J}_{k-1}\neq j \}, \quad j\in \cS,$$
	fulfil
	$$\tilde{T}_{\tilde{\tau}_n(j)}=\tau_n(j), \quad n\in\NN, j\in\cS.$$
	Further $\tilde{J}$ inherits the positive recurrency from $J$ and is necessarily aperiodic whenever $Y^{b,\eta}_t\not \equiv 0$ which implies that $\tilde{J}$ has positive probability to stay in some state. If $Y^{b,\eta}_t \equiv 0$ and $\tilde{J}$ could be periodic, we artificially add a positive probability to stay in some state(s) and take corresponding extra jump times into account. Thus w.l.o.g. $\tilde{J}$ is aperiodic and therefore ergodic and its stationary law $\tilde{\pi}$ is equivalent to $\pi$. \\
	In our setting $\PP_{\tilde{\pi}}(A_n=0)=0$ and $\PP_{\tilde{\pi}}(B_n=0)<1$ are clearly fulfilled and we can apply \cite[Thm. 3.1]{alsmeyerbuckmann1} to prove almost sure convergence of $\frace^{(2)}$. Hereby	
	\begin{align}\label{eq_proofs2} \lim_{n\to \infty} \prod_{k=1}^{\tilde{\tau}_n(j)} A_k =\lim_{n\to \infty} \exp(-(\xi_{\tilde{T}_{\tilde{\tau}_n(j)}-}- \xi_{0})) = \lim_{n\to \infty} \exp(-\xi_{\tau_n(j)} )=0 \quad \PP_{\tilde{\pi}}\text{-a.s.}\end{align} 
	holds since $0<\kappa_\xi<\infty$ implies $\xi_t\to\infty$ $\PP_\pi$-a.s., and due to the recurrency of $(J_t)_{t\geq 0}$ we have $\lim_{n\to \infty} \tau_n(j)=\infty$ $\PP_\pi$-a.s. It remains to show equivalence of \eqref{eq_proofs4} and the second line of \eqref{eq_AB3}, which in our setting reads
	\begin{align}\label{eq_proofs3} \int_{(1,\infty)} \frac{\log q}{\int_{(0,\log q)} \PP_j( \xi_{\tau_1(j)}  >x) dx}  \PP_j(W_j\in dq) <\infty \text{ for some }j \in \cS, \end{align}
	where from \eqref{eq_AB5}
	\begin{align*}W_j &= \max_{1\leq k\leq \tilde{\tau}_1(j)}|\prod_{\ell =1}^{k} A_\ell B_k |  = \max_{1\leq k\leq \tilde{\tau}_1(j)}   e^{-\xi_{\tilde{T}_k-}}  | \Delta(Y^{b,\eta}_{\tilde{T}_k}+ Y^{(2)}_{\tilde{T}_k}) | 
	= \sup_{0< t\leq \tau_1(j)} e^{-\xi_{t-}} | \Delta(Y^{b,\eta}_t+ Y^{(2)}_t) |
	.\end{align*}
	As $0<\kappa_\xi<\infty$, by dominated convergence
	\begin{align*}
	\int_{(0,\log q)} \PP_j( \xi_{\tau_1(j)}  >x) dx = \EE_j\left[ \xi_{\tau_1(j)}^+ \wedge \log q \right] &\overset{q\to\infty} \longrightarrow \EE_j\left[ \xi_{\tau_1(j)}^+ \right] \geq \EE_j\left[ \xi_{\tau_1(j)} \right],
	\end{align*}
	where by \cite[Eq. (10)]{alsmeyerbuckmann2}
	$$\EE_j\left[ \xi_{\tau_1(j)} \right] =  \EE_\pi[\xi_{T_1}] \EE_j\left[ N(j) \right],$$
	with $N(j)\in \NN$ such that $T_{N(j)}=\tau_1(j)$. Applying Wald's equality twice yields
	$$\EE_\pi[\xi_{T_1}]\EE_j\left[ N(j) \right] = \Big(\sum_{j\in\cS} \pi_j \Big(\EE[\xi_1^{(j)}]  + \sum_{\substack{i\in\cS \\ i\neq j}} q_{j,i}\int_\RR xdF^{(j,i)}_\xi(x) \Big)\Big) \EE[T_1]\EE_j\left[ N(j) \right] = \kappa_\xi \EE_i\left[ \tau_1(j) \right]>0,$$
	and hence the denominator in the integral in \eqref{eq_proofs3} has a uniform upper bound and a uniform lower bound which is strictly positive. Thus it can be ignored. \\
	To prove convergence in $\PP_j$-probability of $\frace^{(2)}$ we apply \cite[Thm. 3.4]{alsmeyerbuckmann1} on the Markov modulated perpetuity \eqref{eq_proofsperp} and recall that the non-degeneracy condition \eqref{eq_AB2} and \eqref{eq_proofs2} hold under the given conditions. It remains to show equivalence of the second line of \eqref{eq_AB7} to \eqref{eq_proofs5} which can be done by the same arguments as in the case of almost sure convergence.  That this convergence in $\PP_j$-probability implies convergence in $\PP_{j'}$-probability for all $j'$ follows as in Proposition \ref{prop-weak}.
\end{proof}

\begin{remark}
	Alternatively to the given proof of Proposition \ref{prop_frace2conv} one could also apply Theorem \ref{thmnessandsuff} in the case $\eta_t=Y_t^{b,\eta}+Y_t^{(2)}$ to obtain similar conditions. We decided for a direct approach here as our resulting proofs were slightly shorter. The same is valid for Proposition \ref{prop_frace2conv2} below.
\end{remark}

In case of a finite state space $\cS$ we can also show  conditions for convergence of $\frace^{(2)}$ for infinite, well-defined $\kappa$. Observe that for finite state spaces and assuming non-degeneracy the two types of convergence are equivalent as stated in part (iii) of the following proposition.

To formulate our conditions we introduce 
\begin{equation}\label{eq_main2}
\bar{A}_{\xi}(x):= \sum_{j\in\cS} \pi_j \Big( \gamma_{\xi^{(j)}} + \bar{\nu}_{\xi^{(j)}}^+(1) + \int_1^x \bar{\nu}_{\xi^{(j)}}^+(y) dy + \sum_{\substack{i\in\cS \\i\neq j}} q_{i,j} \int_{\RR_+} y F_\xi^{(i,j)}(dy) \Big),
\end{equation}
which is in spirit of $A_\xi$ and $A_\xi^j$ used in Sections \ref{S1a} and \ref{S3}, yet different.

\begin{proposition}\label{prop_frace2conv2}
	Assume $\cS$ is finite and $\kappa_\xi>0$.
	\begin{enumerate}
		\item The exponential integral $\frace^{(2)}(t)$ converges $\PP_\pi$-a.s. to a finite random variable as $t\to \infty$ if and only if
		\begin{equation}\label{eq-condfrace2finiteas}
		\int_{(1,\infty)} \frac{\log q}{\bar{A}_\xi(\log q)} \PP_j\Big(\sup_{0< t\leq \tau_1(j)} e^{-\xi_{t-}} | \Delta(Y^{b,\eta}_t+ Y^{(2)}_t) |\in dq\Big)<\infty \quad \text{for all} \quad j\in \cS.
		\end{equation}
		\item The exponential integral $\frace^{(2)}(t)$ converges in $\PP_j$-probability to some random variable $\frace^{(2)}_\infty$ as $t\to\infty$,
		 if and only if \begin{equation} \label{eq-condfrace2finiteweak}
		\int_{(1,\infty)} \frac{\log q}{\bar{A}_\xi(\log q)} \PP_j\Big(\Big|\int_{(0,\tau_1(j)]} e^{-\xi_{t-}} d(Y^{b,\eta}_t+ Y^{(2)}_t) \Big| \in dq\Big)<\infty. 	\end{equation} 
		Moreover, in this case convergence in $\PP_j$-probability holds for all $j\in \cS$.
		\item Given
		\begin{equation}\label{eq_degexp2}
		\PP_\pi\left( \frace^{(2)}(t) = c_{J_0}- c_{J_t} e^{-\xi_{t}} \quad \text{for all } t\geq 0\right)<1 
		\end{equation}
		for all sequences $\{c_i,i\in\cS\}$ in $\RR$, the exponential integral $\frace^{(2)}(t)$ converges $\PP_\pi$-a.s. as $t\to \infty$ if and only if it converges in $\PP_j$-probability for some/all $j\in\cS$.
	\end{enumerate}
\end{proposition}
\begin{proof}
	We use the same notation as in the proof of Proposition \ref{prop_frace2conv} and follow its lines up to proving that \eqref{eq_proofs3} is equivalent to \eqref{eq-condfrace2finiteas}.\\
	Note that  $W_i=\sup_{0< t\leq \tau_1(i)} e^{-\xi_{t-}} | \Delta(Y^{b,\eta}_t+ Y^{(2)}_t) |$ as shown in the proof of Proposition \ref{prop_frace2conv} and thus the two expressions only differ in the appearing denominator.\\
	If $\kappa_\xi<\infty$, then as shown in the proof of Proposition \ref{prop_frace2conv} the denominator appearing in  \eqref{eq_proofs3} can be ignored. The same holds true under this assumption for the denominator in \eqref{eq-condfrace2finiteas} as
	\begin{align*}
	\bar{A}_\xi(\log q)
	&\underset{q\to\infty} \nearrow  \sum_{j\in\cS} \pi_j  \left(\gamma_{\xi^{(j)}} + \bar{\nu}_{\xi^{(j)}}^+(1) + \int_1^{\infty} \bar{\nu}_{\xi^{(j)}}^+(y) dy + \sum_{\substack{i\in\cS \\i\neq j}} q_{i,j} \int_{\RR_+} y F_\xi^{(i,j)}(dy)   \right) \\
	&= \kappa_\xi - \underbrace{\sum_{j\in\cS} \pi_j  \int_{(-\infty,-1]} x \nu_{\xi^{(j)}} (dx)}_{\in(-\infty,0]} \in (0,\infty).
	\end{align*}
	Thus assume $\kappa_\xi=\infty$ such that $\xi_t$ tends to $\infty$. 
Let $\{\breve{T}_n, n\in\NN_0\}$ be the jump times of $(Y^{b,\xi}_t+J_t)_{t\geq 0}$ with  $\breve{T}_0:=0$ such that  $\{T_n,n\in\NN_0\} \subseteq \{\breve{T}_n, n\in\NN_0\}$ and $\{\breve{T}_n, n\in\NN_0\}$ contains all jump times of $(Y^{b,\xi}_t+X^{(2)}_t)_{t\geq 0}$. Repeating the computation and argumentation leading to \eqref{eq_proofsperp} we note that this generates another retarded discrete time version $\breve{J}$ of $J$ which is w.l.o.g. aperiodic and ergodic with stationary law $\breve{\pi}$ equivalent to $\pi$ and such that its return times $\breve{\tau}_n(j)$ satisfy $\breve{T}_{\breve{\tau}_n(j)}=\tau_n(j)$. \\
	Then 
	\begin{align*}
	\int_{(0,\log q)} \PP_j(\xi_{\tau_1(j)}>x)dx &= \EE_j \left[\xi_{\tau_1(j)}^+ \wedge \log q \right]
	\asymp \EE_\pi[\xi_{\breve{T}_1}^+ \wedge \log q] \quad \text{for }q\to \infty,
	\end{align*}
	by \cite[Lemma 8.16]{alsmeyerbuckmann2}, where we use the notation of $f(x)\asymp g(x)$ whenever $\liminf_{x\to\infty} \frac{f(x)}{g(x)}>0$ and $\limsup_{x\to\infty} \frac{f(x)}{g(x)}<\infty$.
	Further
	\begin{align*}
	\EE_\pi[\xi_{\breve{T}_1}^+ \wedge \log q] 
	&= \int_{(0,\log q)} \PP_\pi(\xi_{\breve{T}_1}>x) dx\\
	&=  \int_{(0,\log q)} \PP_\pi\left(\gamma^\xi_{\breve{T}_1} + W^\xi_{\breve{T}_1} + Y^{b,\xi}_{\breve{T}_1} + Y^{s,\xi}_{\breve{T}_1}  +  X_{\breve{T}_1}^{(2)} >x\right) dx,
	\end{align*}
	where  $(W^\xi_t +  Y^{s,\xi}_t)_{t\geq 0}$ is a martingale and $\sup_{j\in\cS} |\gamma_{\xi^{(j)}}|<\infty$ for finite $\cS$. Thus 
	$$	\int_{(0,\log q)} \PP_\pi\left(\gamma^\xi_{\breve{T}_1} + W^\xi_{\breve{T}_1} + Y^{b,\xi}_{\breve{T}_1} + Y^{s,\xi}_{\breve{T}_1}  +  X_{\breve{T}_1}^{(2)} >x \right)  dx \asymp 	\int_{(0,\log q)} \PP_\pi\left( Y^{b,\xi}_{\breve{T}_1}  +  X_{\breve{T}_1}^{(2)} > x\right) dx \quad \text{as }q\to \infty,$$
	and hence 
	\begin{align*}
	\int_{(0,\log q)} \PP_j(\xi_{\tau_1(j)}>x)dx 
	&\asymp \int_{(0,\log q)} \PP_\pi\left( Y^{b,\xi}_{\breve{T}_1}  +  X_{\breve{T}_1}^{(2)} > x\right) dx\\
	&= \EE_\pi[(Y^{b,\xi}_{\breve{T}_1}  +  X_{\breve{T}_1}^{(2)})^+\wedge \log q]\\
	&= \sum_{j\in\cS} \pi_j \left(\int_{(1,\log q)} \bar{\nu}^+_\xi(y) dy + \sum_{\substack{i\in\cS \\i\neq j}} q_{i,j} \int_{(0,\log q)} y F_\xi^{(i,j)}(dy) \right) \\
	&\asymp \bar{A}_\xi(\log q),	
	\end{align*}
	which implies equivalence of \eqref{eq_proofs3} and \eqref{eq-condfrace2finiteas}. \\
	Again, the proof for convergence in $\PP_j$-probability can be carried out analogously. That this implies convergence in $\PP_j$-probability for all $j$ follows as in Proposition \ref{prop-weak}.\\
	Finally (iii) follows from \cite[Rem. 3.8]{alsmeyerbuckmann1} and applying the results from Section \ref{s41} on $\frace^{(2)}$.
\end{proof}

\begin{remark}
	Note that \eqref{eq_degexp2} excludes degeneracy of $\frace^{(2)}$, but this does not necessarily imply non-degeneracy of $\frace_{(\xi,\eta)}$ or vice versa. This would only be the case if one assumes additionally that $\frace^{(1)}$ is degenerate, i.e. if there exists a sequence $\{\tilde{c}_j, j\in\cS\}$ such that
	\begin{equation}\label{eq_degexp1}
	\frace^{(1)}=\tilde{c}_{J_0}- \tilde{c}_{J_t} e^{-\xi_t} \quad \PP_\pi\text{-a.s. for all }t\geq 0.
	\end{equation}
	Indeed, given \eqref{eq_degexp1}, \eqref{eq_deg00} is equivalent to the existence of a (unique) sequence $\{\check{c}_j,j\in\cS\}$ such that 
	$$\frace^{(2)}=\check{c}_{J_0}- \check{c}_{J_t} e^{-\xi_t} \quad \PP_\pi\text{-a.s. for all }t\geq 0,$$
	can be seen by direct computations.
\end{remark}

The following corollary exemplarily summarizes results of Propositions \ref{prop_frace1strongconv} and  \ref{prop_frace2conv}. Similar statements for other scenarios can easily be formulated using the above statements. 

\begin{corollary} Assume $0<\kappa_\xi<\infty$ and \eqref{eq_proofs0} as well as \eqref{eq_proofs4} hold. Then $\frace_{(\xi,\eta)}(t)$ converges $\PP_\pi$-a.s. to a finite random variable as $t\to\infty.$	
\end{corollary}

\section*{Acknowledgements}

We would like to thank two anonymous referees for their comments on an earlier version of this paper which helped us to improve it. 
	
  \bibliography{literatureMAPfunc}
	  
 \end{document}